\documentclass[reqno]{amsart}

\usepackage{amsmath,amssymb}
\usepackage{amsfonts, amscd, epsfig, amsmath, amssymb,enumerate}
\usepackage{amstext}
\usepackage{graphicx}
\usepackage[usenames,dvipsnames,svgnames,table]{xcolor}
\usepackage{mathrsfs}
\usepackage{mathptmx}       
\usepackage{helvet}         
\usepackage{courier}        

\usepackage{tikz}
\usetikzlibrary{backgrounds}
\usetikzlibrary{patterns,fadings}
\usetikzlibrary{arrows,decorations.pathmorphing}
\usetikzlibrary{calc}
\definecolor{light-gray}{gray}{0.95}
\usepackage{float}
\usepackage[colorlinks=true,linkcolor=blue,citecolor=magenta]{hyperref}

\newtheorem{theorem}{Theorem}[section]
\newtheorem{lemma}[theorem]{Lemma}
\newtheorem{proposition}[theorem]{Proposition}
\newtheorem{corollary}[theorem]{Corollary}
\newtheorem{remark}[theorem]{Remark}
\newtheorem{definition}{Definition}[section]

\newcommand{\bola}{\textrm{\Large\textperiodcentered}}

\numberwithin{equation}{section}

\newcommand{\mc}[1]{{\mathcal #1}}

\newcommand{\mb}[1]{{\mathbf #1}}
\newcommand{\bb}[1]{{\mathbb #1}}

\renewcommand{\epsilon}{\varepsilon}

\newcommand{\vphi}{\varphi}
\newcommand{\p}{\partial}
\def\centerarc[#1](#2)(#3:#4:#5){\draw[#1] ($(#2)+({#5*cos(#3)},{#5*sin(#3)})$) arc (#3:#4:#5);}

\renewcommand\vec[1]{\overrightarrow{#1}}
\newcommand\vecleft[1]{\overleftarrow{#1}}

\begin{document}

\title[Crossover to the Stochastic Burgers  equation for the WASEP with a slow bond]{Crossover to the Stochastic Burgers  equation\\ for the WASEP with a slow bond}

\author{Tertuliano Franco}
\address{UFBA\\
 Instituto de Matem\'atica, Campus de Ondina, Av. Adhemar de Barros, S/N. CEP 40170-110\\
Salvador, Brazil}
\curraddr{}
\email{tertu@ufba.br}
\thanks{}

\author{Patr\'{\i}cia Gon\c{c}alves}
\address{\noindent Departamento de Matem\'atica, PUC-RIO, Rua Marqu\^es de S\~ao Vicente, no. 225, 22453-900, Rio de Janeiro, Rj-Brazil and CMAT, Centro de Matem\'atica da Universidade do Minho, Campus de Gualtar, 4710-057 Braga, Portugal.}
\email{patricia@mat.puc-rio.br}

\author{Marielle Simon}
\address{\noindent \'Equipe MEPHYSTO, Inria Lille -- Nord Europe, 40 avenue du Halley, 59650 Villeneuve d'Ascq (France)}
\email{marielle.simon@inria.fr}


\begin{abstract}
We consider the weakly asymmetric simple exclusion process in the presence of a slow bond and starting from the invariant state, namely the Bernoulli product measure of parameter $\rho\in(0,1)$. The rate of passage of particles to the right (resp. left) is $\frac1{\vphantom{n^\beta}2}+\frac{a}{2n^{\vphantom{\beta}\gamma}}$ (resp. $\frac1{\vphantom{n^\beta}2}-\frac{a}{2n^{\vphantom{\beta}\gamma}}$) except at the bond of vertices $\{-1,0\}$ where the rate to the right (resp. left) is given by $\frac{\alpha}{2n^\beta}+\frac{a}{2n^{\vphantom{\beta}\gamma}}$ (resp. $\frac{\alpha}{2n^\beta}-\frac{a}{2n^{\vphantom{\beta}\gamma}}$). Above, $\alpha>0$,  $\gamma\geq \beta\geq 0$, $a\geq 0$.  For $\beta<1$, we show that the limit density fluctuation field is an  Ornstein-Uhlenbeck process defined on the Schwartz space if $\gamma>\frac12$, while for $\gamma = \frac12$ it is an energy solution of the stochastic Burgers equation. For  $\gamma\geq\beta=1$, it is an Ornstein-Uhlenbeck  process associated to the heat equation with Robin's boundary conditions.  For $\gamma\geq\beta> 1$, the limit density fluctuation field is an Ornstein-Uhlenbeck  process associated to the heat equation with Neumann's boundary conditions.
\end{abstract}

\maketitle


\section{Introduction}

Over the last decades, there has been an intense research activity in the derivation of macroscopic laws from suitable underlying stochastic microscopic dynamics. For interacting particle systems of exclusion type,  the scenario is more or less well understood as soon as the jump rates are symmetric (see \cite{kl,Spohn}), but for weakly asymmetric systems only partial answers have been given, and there are still challenging behaviors to establish. Even harder is the derivation of macroscopic laws for microscopic dynamics with local defects. By this we mean that the microscopic particle system is locally perturbed,  and depending on the type of perturbation, the macroscopic laws can hold different boundary conditions.

In this paper we consider the simplest microscopic dynamics of exclusion type: we  add a weak asymmetry, and we also perturb the dynamics at one particular bond. More precisely, our interest focuses on  establishing the crossover of the  equilibrium density fluctuations for the weakly asymmetric  simple exclusion process (WASEP) with strength asymmetry
$an^{2-\gamma}$ (with $a\geq 0$ and  $\gamma\geq\frac{1}{2}$) and with a jump rate at the bond $\{-1,0\}$ that is slower than the rate at other bonds.

Let us go into more details:  particles are distributed on the line $\bb Z$, with the condition that at most one particle per site is allowed. The dynamics can be described as follows: particles at different sites wait independent exponential times; when a clock rings, a particle at the site $x$ jumps to $x+1$ (resp. $x-1$) at rate $\frac{1}{\vphantom{n^\beta}2}+\frac{a}{2n^{\vphantom{\beta}\gamma}}$ (resp. $\frac{1}{\vphantom{n^\beta}2}-\frac{a}{2n^{\vphantom{\beta}\gamma}}$), but the jump rate from $-1$ to $0$ (resp. $0$ to $-1$) is equal to $\frac{\alpha}{2n^\beta}+\frac{a}{2n^{\vphantom{\beta}\gamma}}$ (resp. $\frac{\alpha}{2n^\beta}-\frac{a}{2n^{\vphantom{\beta}\gamma}}$), with $\alpha> 0$ and $\beta\geq 0$; see Figure \ref{fig5} below. In order to have positive rates we have to impose some conditions on the parameters $a,\beta,\alpha,\gamma$, see \eqref{eq:hyp}.

\begin{figure}[H]
\centering
\begin{tikzpicture}
\centerarc[thick,<-](2.5,0.3)(10:170:0.45);
\centerarc[thick,->](2.5,-0.3)(-10:-170:0.45);
\centerarc[thick,->](4.5,-0.3)(-10:-170:0.45);
\centerarc[thick,<-](4.5,0.3)(10:170:0.45);
\centerarc[thick,->](6.5,-0.3)(-10:-170:0.45);
\centerarc[thick,<-](6.5,0.3)(10:170:0.45);

\draw (0,0) -- (9,0);

\shade[ball color=black](1,0) circle (0.25);
\shade[ball color=black](3,0) circle (0.25);
\shade[ball color=black](5,0) circle (0.25);
\shade[ball color=black](8,0) circle (0.25);
\shade[ball color=black](9,0) circle (0.25);

\filldraw[fill=white, draw=black]
(0,0) circle (.25)
(2,0) circle (.25)
(4,0) circle (.25)
(6,0) circle (.25)
(7,0) circle (.25);

\draw (2.6,-0.1) node[anchor=north] {\small $-2$}
(3.6,-0.1) node[anchor=north] {\small $-1$}
(4.6,-0.1) node[anchor=north] {\small $0$}
(5.6,-0.1) node[anchor=north] {\small $1$}
(6.6,-0.1) node[anchor=north] {\small $2$};
\draw (2.5,0.8) node[anchor=south]{\small $\frac{1}{2}\!+\!\frac{a}{2n^\gamma}$};
\draw (2.5,-0.8) node[anchor=north]{\small $\frac{1}{2}\!-\!\frac{a}{2n^\gamma}$};
\draw (6.5,0.8) node[anchor=south]{\small $\frac{1}{2}\!+\!\frac{a}{2n^\gamma}$};
\draw (6.5,-0.8) node[anchor=north]{\small $\frac{1}{2}\!-\!\frac{a}{2n^\gamma}$};
\draw (4.5,0.8) node[anchor=south]{\small $\frac{\alpha}{2n^\beta}\!+\!\frac{a}{2n^\gamma}$};
\draw (4.5,-0.8) node[anchor=north]{\small $\frac{\alpha}{2n^\beta}\!\!-\!\!\frac{a}{2n^\gamma}$};
\end{tikzpicture}
\caption{Illustration of the jump rates. The bond of vertices $\{-1,0\}$ has particular rates associated to it.}\label{fig5}
\end{figure}
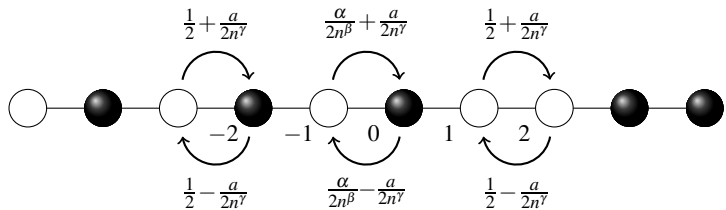
For the choice $a=\beta=0$ and $\alpha=1$, we recover the symmetric simple exclusion process (SSEP), which has been  deeply investigated  in the literature and whose density fluctuations  are given by an Ornstein-Uhlenbeck (OU)  process (see \cite{Ravi1992}), and consequently the fluctuations are Gaussian. The choice $\alpha=\gamma=1$ and $\beta=0$ corresponds to the WASEP whose equilibrium density fluctuations were studied in \cite{DMPS} and the non-equilibrium fluctuations were studied in \cite{DG}. For this regime of the strength asymmetry, namely $\gamma=1$, the limiting process is an OU process taking values in the Schwartz space and consequently the fluctuations are Gaussian again. The only difference with respect to the SSEP limit is that this OU process has a drift term (coming from the asymmetric part of the dynamics) which can be removed by taking the system in a reference frame, or else, by performing a Galilean transformation of the system, from where we recover exactly the OU limit of the SSEP. When removing the drift to the system, there is no effect of the asymmetry, and therefore one has to strengthen the asymmetry by decreasing the value of $\gamma$. In this very same regime, but for a stronger asymmetry, namely $\gamma=\frac{1}{2}$, the non-equilibrium fluctuations  were derived in \cite{BerG}, being the limiting process a solution of the stochastic Burgers  equation (SBE). In this case, the strong asymmetry gives rise to a non-linear term in the stochastic partial differential equation. Finally, the crossover for the equilibrium density fluctuations, in the regime $\gamma\in[\frac{1}{2},1]$, has been established in \cite{gj2012,gj2014}. More precisely, for $\gamma\in{(\frac{1}{2},1]}$ the limit is the same OU process as in the SSEP limit,  and for this reason the process belongs to the Edwards-Wilkinson \cite{EW} universality class, but for $\gamma=\frac{1}{2}$ the limit is a solution of the SBE and the system belongs to the Kardar-Parisi-Zhang (KPZ) \cite{KPZ} universality class.

  In all the previous results the slow bond was not taken into account. The SSEP with a slow bond has been investigated in \cite{fgn3} without the weak asymmetry. In that paper the equilibrium density fluctuations were derived for $\alpha>0$, $\beta\geq 0$ and $a=0$. The authors proved a phase transition depending on the regime of $\beta$: for $\beta<1$ the limit is the same  OU process as in the SSEP limit, for $\beta=1$ the limit is an OU process defined on a Fr\'echet space in which the functions have a boundary condition of Robin's type; and for $\beta>1$ the limit is an OU process defined on a Fr\'echet space in which the functions  have a boundary condition of Neumann's type.

In this paper we superpose all these dynamics, since we perturb the SSEP by a weak asymmetry of strength $an^{2-\gamma}$ (with $a\geq 0$ and $\gamma\geq \frac{1}{2}$) and we introduce a slow bond at $\{-1,0\}$.  The system is taken under the invariant state, namely the Bernoulli product measure of parameter $\rho\in{(0,1)}$ that we denote by $\nu_\rho$. We also take it in a reference frame so that we do not see the transport behavior of the system. In order to make the presentation simpler, we choose $\rho=\frac{1}{2}$ for which the transport velocity is zero and the Galilean transformation is not necessary. Nevertheless,  all our results hold for other values of $\rho$. Moreover, we point out that the strength of importance is  in fact $\gamma\in [1/2,2]$, because for $\gamma>2$ the asymmetry strength $n^{2-\gamma}$ goes to zero. Since the proofs stand also for $\gamma>2$ we state the results in the general setting.

We also underline that by perturbing microscopically the system, one can lose nice properties on its invariant states, as it is the case, for example, for the totally asymmetric simple exclusion process, see \cite{Timo2001}.  For that model the invariant states are no longer the Bernoulli product measures and for that reason the asymptotic behavior of the system is very hard to derive. However in our case, and as it happens for the models of \cite{fgn3}, even with the weak asymmetry, the Bernoulli product measures are still invariant. This point is crucial in what follows. We derive the density fluctuations of the system and we prove the crossover from the Edwards-Wilkinson universality class to the KPZ universality class as in \cite{gj2014,GJS} for the regime $\beta<1$; in the other cases we obtain the limiting processes as in  \cite{fgn3} (see Figure \ref{fig:energy}).

\begin{figure}[!h]
\centering
\begin{tikzpicture}

\begin{scope}[xshift=2cm]
\fill[fill=black!20!white] (0,1)--(1,1)--(1.98,2)--(1.98,3.85)-- (0,3.85)--cycle;
\fill[pattern=dots] (2,2)--(2,3.9)--(3.9,3.9)--cycle;
\draw[ultra thick] (2,2)--(2,3.9);
\draw[ultra thick, color=white] (1,1)--(0,1);
\draw[ultra thick, dash pattern=on 3.5pt off 1.5pt] (1,1)--(0,1);
\filldraw[fill=white, draw=black, thick] (2,2) circle (2pt);
\draw[->] (-0.5,0)--(4.5,0) node[below]{$\beta$};
\draw[->] (0,-0.5)--(0,4.2) node[left]{$\gamma$};
 \draw (2pt,2) -- (-2pt,2) node[left]{$1$};
 \draw (2,2pt) -- (2,-2pt) node[below]{$1$};
 \draw (1,2pt)--(1,-2pt) node[below]{$\frac{1}{2}$};
 \draw (0,1) node[left]{$\frac{1}{2}$};
\end{scope}

\draw[ultra thick, dash pattern=on 3.5pt off 1.5pt] (0,-1.5)--(1,-1.5) node[right]{Stochastic Burgers equation (KPZ regime)};
\fill[black!20!white] (0,-2) rectangle (1,-2.5);
\draw (1,-2.25) node[right]{OU process with no boundary conditions};
\draw[ultra thick] (0,-3)--(1,-3) node[right]{OU process with Robin's boundary conditions};
\fill[pattern=dots] (0,-3.5) rectangle (1,-4);
\draw (1,-3.75) node[right]{OU process with Neumann's boundary conditions};
\filldraw[fill=white, draw=black, thick] (0.5,-4.5) circle (2pt) ;
\draw (1,-4.5) node[right]{OU process with Robin's boundary conditions and stronger noise};
\end{tikzpicture}
\caption{Macroscopic density fluctuations.
In the region $\beta>\gamma$ the process is not defined (negative rates).}
\label{fig:energy}
\end{figure}
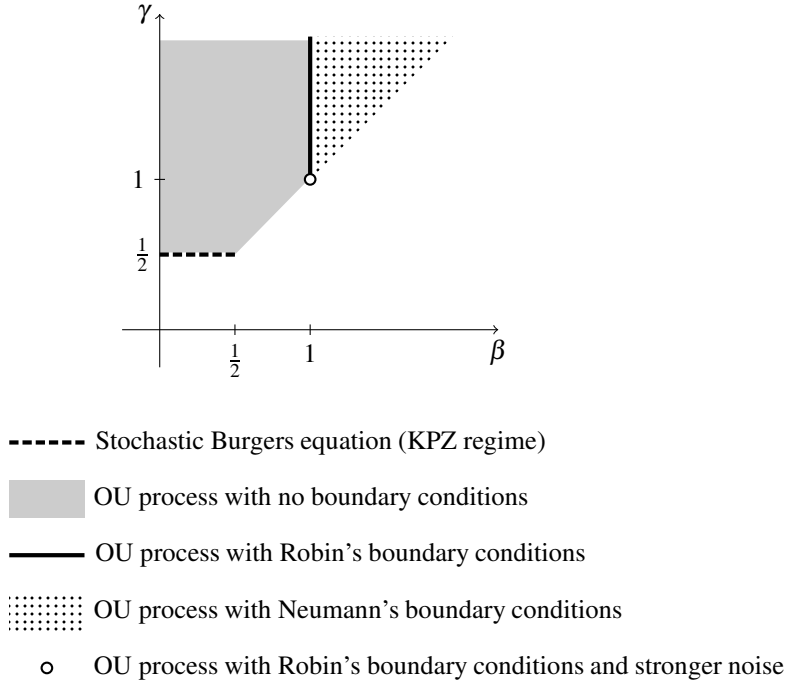

The structure of the proof is standard and consists in showing  tightness plus the  characterization of the limit points by a martingale problem. However, the main argument of the proof needs a very careful investigation: more precisely, the derivation of a second order Boltzmann-Gibbs principle, as stated in \cite[Theorem 7]{gj2014} and \cite[Theorem 3.2]{GJS}, is challenging. This principle allows to replace certain additive functionals of local functions of the dynamics by additive functionals given in terms of the density of particles. However, the ideas of \cite{gj2014,GJS}  do not apply to the model considered here, since our rates are not bounded from below and the usual spectral gap inequality is unknown.  In this paper we expose a new way to estimate the error in the replacement performed in the Boltzmann-Gibbs principle which avoids the spectral gap inequality. Our new argument  consists in splitting the asymmetric part of the current in certain local functions of the dynamics. In each one of these functions we are able to replace the occupation site variable by its average on  big boxes, through a multi-scale analysis as in the spirit of \cite{G2008,gj2014,GJS}. The presence of the slow bond makes this analysis more complicated but we are able to control the errors in such a way that   we recover the same behavior of the system as if there was no slow bond.

The  method presented here is not limited to our model and can be applied to other works: for instance, with our approach we are able to recover the results of \cite{G2008},  where it has been proved that the density fluctuation field for the asymmetric simple exclusion process does not evolve up to the time scale $tn^{\frac43}$. More remarkably, we recover the same behavior for the energy fluctuation field in the one-dimensional Hamiltonian system with exponential interactions considered in \cite{BG}. In that paper, to overcome the non-existence of the spectral gap, Bernardin and Gon\c calves propose an alternative approach which is based on duality properties in Hilbert spaces and computations of some resolvent norms, and is more complicated than the direct estimates that we give  here. We also emphasize that our method can be used to derive the density fluctuations for more general interacting particle systems for which the spectral gap inequality is not known, like exclusion processes, zero-range processes and stochastically perturbed Hamiltonian systems with polynomial potentials. Let us remark  that for stochastically perturbed Hamiltonian systems with general interactions, to our knowledge, the only known result on the energy fluctuations is the work of \cite{OllaMakiko} where the stochastic noise has to be strong enough in order to get the correct bound for the spectral gap inequality. In \cite{BGJ,BGJSS} it was considered an harmonic potential with an exchange noise, and the non-existence of the spectral gap was overcome by using the structure of the invariant measure and the linearity of the deterministic dynamics which permit to use  Fourier transforms.

We believe that our technique can be carried out for proving new results: for instance, in the class of Hamiltonian systems without the spectral gap property, one could work out the case of polynomial potentials or, at least, some small perturbation of the harmonic potential; and furthermore, one could investigate the class of kinetically constrained lattice gases which have been introduced and intensively studied in the literature, see \cite{GLT,GJS}. For these models, one should be able to repeat our multi-scale argument for higher degree polynomial functions,  as done in \cite{GJCPAM}. These are works in progress.

Finally, for the model considered in this paper, we  could also derive the equilibrium fluctuations for the current of particles as in \cite{gj2014} by relating the density fluctuation field with the current of particles and we would get the same results as if there was no slow bond, see \cite{gj2014,fgn3} for details.

Here follows an outline of the paper. In Section \ref{sec:model} we first introduce the  model, we define the Fr\'echet spaces where the density fluctuations fields are defined and study the invariant measures. Section \ref{sec:results} is devoted to stating the main results, namely the second order Boltzmann-Gibbs principle (Theorem \ref{theo:BG}),  the convergence to the OU process (Theorem \ref{thm:crossover}) depending on the values of the parameters and the crossover to the SBE (Theorem \ref{theo:fluct}).  The derivation of the equilibrium density fluctuations is detailed in Sections \ref{sec:comp} and \ref{sec:proof}. The second order Boltzmann-Gibbs principle is entirely proved in Section \ref{sec:BG}, and then, in Section \ref{sec:auxi}, we use the new techniques developed in the latter section in order to obtain two  auxiliary estimates.  Appendix \ref{app} provides some necessary conditions on the semi-groups associated to the macroscopic fluctuations.

\section{Notations and definitions}\label{sec:model}

\subsection{The model}
Let $n$ be a positive integer and $\Omega = \{0,1\}^{\bb Z}$ be the state space of the Markov process  $\{\eta_t\,; \, t \geq 0\}$ whose dynamics can be entirely defined by its infinitesimal generator $\mathcal{L}_n$. The latter is defined on local functions $f: \Omega \to \bb R$ by
\begin{equation*}
\mathcal L_nf(\eta) = \sum_{x \in \bb Z}\xi^n_{x,x+1}(\eta) \nabla_{x,x+1} f(\eta)+\xi^n_{x,x-1}(\eta) \nabla_{x,x-1} f(\eta)
\end{equation*}
where for $\eta\in\Omega$ and $x,y\in\mathbb Z$, we denote $\nabla_{x,y} f(\eta)=f(\eta^{x,y})-f(\eta)$ and $\eta^{x,y} \in \Omega$ is defined as
\[
\eta^{x,y}(z)=
\left\{
\begin{array}{rl}
\eta(y);& z=x\\
\eta(x);& z=y\\
\eta(z);& z\neq x,y.\\
\end{array}
\right.
\]
The rates introduced above are chosen in the following way: for $x\neq{-1}$,
\[
  \xi^n_{x,x+1}(\eta) =\Big(\frac{1}{2} + \frac{a}{2 n^{\gamma}}\Big) \; \eta(x)(1-\eta(x+1)),
\]
\[
  \xi^n_{-1,0}(\eta) =\Big(\frac{\alpha}{2n^\beta} + \frac{a}{2 n^{\gamma}}\Big)\; \eta(-1)(1-\eta(0)),
\]
and for $x\neq 0$,
\[
  \xi^n_{x,x-1}(\eta) =\Big(\frac{1}{2} - \frac{a}{2 n^{\gamma}}\Big) \; \eta(x)(1-\eta(x-1)),
\]
\[
  \xi^n_{0,-1}(\eta) =\Big(\frac{\alpha}{2n^\beta} - \frac{a}{2 n^{\gamma}}\Big)\; \eta(0)(1-\eta(-1))\,,
\]
where $\alpha >0$, $a\geq 0$, $\beta\geq 0$, $\gamma \geq \frac{1}{2}$ are real parameters. See Figure \ref{fig5} for an illustration of the dynamics. In order to avoid negative rates, in the following we always assume
\begin{equation}
(\gamma > \beta) \quad \text{ or } \quad (\beta = \gamma\ \text{ and }\ \alpha \ge a). \label{eq:hyp}
\end{equation}  Notice that we allow the case $\gamma=\beta$ and $a=\alpha$. In this situation, the slow bond $\{-1,0\}$ is totally asymmetric with  left rate $\xi_{0,-1}^n$ equal to zero, and right rate $\xi_{-1,0}^n$ vanishing as $n \to\infty$.

Let us fix $T>0$. We are interested in the evolution of this exclusion process in the diffusive time scale, thus we denote by $\{\eta_{tn^2}\,; \, t\in [0,T]\}$ the Markov process on $\Omega$ associated to the generator $n^2 \mathcal{L}_n$. The path space of trajectories which are right-continuous, with  left-limits and taking values in $\Omega$ is denoted by $\mathcal{D}([0,T],\Omega)$. For any initial probability measure $\mu$ on $\Omega$, we denote by $\bb P_{\mu}$ the probability measure on  $\mathcal{D}([0,T],\Omega)$ induced by $\mu$ and the Markov process $\{\eta_{tn^2}\; ;\; t \in [0,T]\}$.

\subsection{Definition of $\mathcal{S}'_\beta(\bb R)$ and of the operators $\nabla_\beta$ and $\Delta_\beta$}

\begin{definition}\label{def:norm}
Fix $\alpha>0$. For any function $\vphi:\bb R\to \bb R$ we define the norm:
\[
\Vert \vphi \Vert_{2,\beta}^2 =  \int_{\bb R} \vphi^2(u) du + \mathbf{1}_{\beta=1} \Big(\frac{\varphi^2(0)}{\alpha}  + \mathbf{1}_{\gamma=1} \frac{a\; \varphi^2(0)}{\alpha^2}\Big).
\]
Let $\mathbf{L}^2_{\beta}(\mathbb{R})$ be the space of functions $\vphi:\bb R\to \bb R$ such that  $\Vert \vphi \Vert_{2,\beta} < +\infty$.
\end{definition}
When $\beta \neq 1$, the norm $\Vert \cdot \Vert_{2,\beta}$ is the usual ${\mb L}^2(\bb R)$-norm with respect to the Lebesgue measure, and for the sake of simplicity we will rewrite it as $\Vert \cdot \Vert_{2}$. Despite the norm above depends on $a, \alpha$ and $\gamma$, for simplicity of notation we do not index on them. Given $\vphi:\bb R \to \bb R$, we denote:
\[
\vphi(0^+):= \lim_{\substack{u\to 0\\u>0}} \vphi(u) \quad \text{ and } \quad \vphi(0^-):= \lim_{\substack{u\to 0\\u<0}} \vphi(u),
\]
whenever these limits exist.

\begin{definition}
We define $\mc S(\bb R \backslash \{0\})$ as the space of functions $\vphi:\bb R\to\bb R$ such that: \begin{enumerate}
\item[(i)]  $\vphi$ is smooth on $\bb R\backslash\{0\}$, i.e.~$\vphi \in \mc C^{\infty}(\bb R\backslash\{0\})$,
\item[(ii)]  $\vphi$ is continuous from the right at 0,
\item[(iii)]  for all non-negative integers $k,\ell$, the function $\vphi$ satisfies
\begin{equation}\label{eq2.2}
\Vert \vphi \Vert_{k,\ell}:= \sup_{u \neq 0} \Big| (1+|u|^\ell) \frac{d^k\vphi}{du^k}(u)\Big| < \infty.
\end{equation}
\end{enumerate}
\end{definition}
\begin{remark}
It is a consequence of \eqref{eq2.2} that  the side limits
\[
\frac{d^k\vphi}{du^k}(0^+)\qquad\textrm{and}\qquad \frac{d^k\vphi}{du^k}(0^-)
\]
exist for any integer $k \ge 0$.
\end{remark}
\begin{definition}\label{def23}
\begin{enumerate}
\item For $\beta <1$, we define
 \[\mc S_\beta(\bb R)\;:=\; \mc S(\bb R \backslash \{0\})\cap \mc C^\infty(\bb R).\] In other words, in this case $\mc S_\beta(\bb R)$  is the usual Schwartz space $\mc S(\bb R)$.
\item For $\beta=1$, we define $\mc S_\beta(\bb R)$ as the subset of $\mc S(\bb R \backslash \{0\})$ composed of functions $\vphi$ such that
\begin{equation} \frac{d^{2k+1}\vphi}{du^{2k+1}}(0^+)= \frac{d^{2k+1}\vphi}{du^{2k+1}}(0^-) = \alpha  \Big(\frac{d^{2k}\vphi}{du^{2k}}(0^+)-\frac{d^{2k}\vphi}{du^{2k}}(0^-)\Big)
\label{eq:bound_beta_1}\end{equation}
for any integer $k\geq 0$.
\item For $\beta >1$, we define $\mc S_\beta(\bb R)$ as the subset of $\mc S(\bb R \backslash \{0\})$ composed of functions $\vphi$ such that \begin{equation*}\frac{d^{2k+1}\vphi}{du^{2k+1}}(0^+)=\frac{d^{2k+1}\vphi}{du^{2k+1}}(0^-)=0
\end{equation*}
for any integer $k\geq 0$.
\end{enumerate}
Finally, let $\mc S'_\beta(\bb R)$ be the topological dual of $\mc S_\beta(\bb R)$.
\end{definition}
\begin{definition}\label{def24}
We define  $\nabla_\beta: \mc S_\beta(\bb R) \to \mc S(\bb R\backslash \{0\})$ and $\Delta_\beta: \mc S_\beta(\bb R) \to \mc S_\beta(\bb R)$ as the operators acting on functions $\vphi \in \mc S_\beta(\bb R)$ as
\[
\nabla_\beta\vphi(u)= \begin{cases} \displaystyle \frac{d\vphi}{du}(u), & \text{ if } u \neq 0,\\ \\ \displaystyle  \frac{d\vphi}{du}(0^+), & \text{ if } u = 0, \end{cases}
\quad \text{and} \quad \Delta_\beta\vphi(u)= \begin{cases}\displaystyle \frac{d^2\vphi}{du^2}(u), & \text{ if } u \neq 0,\\ \\ \displaystyle\frac{d^2\vphi}{du^2}(0^+), & \text{ if } u = 0. \end{cases}
\]
\end{definition}
When $\beta<1$, the operator $\nabla_\beta$ (resp. $\Delta_\beta$) coincides with the usual gradient operator $\nabla$ (resp. Laplacian operator $\Delta$). Notice that the definitions above are closely related to Definitions 2.3--2.6 given in \cite{fgn3}, except that here we add a minor correction. Under the new Definitions \ref{def23} and \ref{def24}, if $\vphi\in \mc S_\beta (\bb R)$, then
\begin{equation}\label{contingency}
\Delta_\beta T_t^\beta \vphi\in \mc S_\beta(\bb R)\,,
\end{equation}
where $T_t^\beta$ is the semi-group of the partial differential equation associated to the macroscopic evolution, as defined in Appendix \ref{app} (see also \cite[Section 2.3]{fgn3}). The inclusion above is required when characterizing the limit points of the processes (see Section \ref{ssec:limitpoints}). We left to Appendix \ref{app} the proof of \eqref{contingency}.

Finally, for $\vphi \in \mc S_\beta(\bb R)$, $x \in \bb Z$ and $n \in \bb N$, we define the discrete approximations of the first and second derivatives of $\vphi$ as follows:
\begin{equation*}
\nabla_x^n\vphi:=n\Big\{\vphi\Big(\frac{x+1}n\Big)-\vphi\Big(\frac{x}n\Big)\Big\}\quad \textrm{and}\quad
\Delta_x^n\vphi:=n^2\Big\{\vphi\Big(\frac{x+1}{n}\Big) -2\vphi\Big(\frac{x}{n}\Big)+\vphi\Big(\frac{x-1}{n}\Big)\Big\}.
\end{equation*}

\subsection{Invariant measures}
Let $\rho \in (0,1)$ and let $\nu_\rho$ be the product Bernoulli measure on $\Omega$ with density $\rho$:
\[
\nu_\rho\{\eta(x)=1\} = \rho.
\]
The  measures $\{\nu_\rho\,;\, \rho \in (0,1)\}$ are invariant but not reversible with respect to the evolution of $\{\eta_t\,; \,t \geq 0\}$.
To assure the invariance, it is enough to check that
\begin{equation}\label{eq2.3}
\int_\Omega \mathcal L_nf(\eta)\, \nu_\rho(d\eta) =0
\end{equation}
for any local function $f:\{0,1\}^{\bb Z}\to \bb R$.
Given a local function $f$, let $L\in \bb N$ be such that $f$ depends only on the occupation of sites $I=\{-L,-L+1,\cdots, L\}$. Therefore, $f$ can be identified with a function $f:\{0,1\}^I\to \bb R$.
Given a set $A$, we denote by $\textbf{1}_A(u)$ the function which equals 1 when $u\in A$, and 0 when $u\notin A$. By the identity
\begin{equation*}
f(\eta)\;=\;\sum_{\tilde{\eta}\in \{0,1\}^I} f(\tilde{\eta}){\bf 1}_{\{\tilde{\eta}\}}(\eta)\,,
\end{equation*}
and observing that \[{\bf 1}_{\{\tilde{\eta}\}}(\eta)= \prod_{x\in\bb Z}\big(1-|\eta(x)-\tilde{\eta}(x)|\big),\] we can rewrite $f$ as a weighted sum of products  of $\eta(x)$ or $1-\eta(x)$. By linearity, in order to obtain \eqref{eq2.3} it is enough to prove it for functions of the form $f(\eta)=\eta(x_1)\cdots\eta(x_k)$, where $x_1,...,x_k$ are integers.  Let us illustrate the case $f(\eta)=\eta(0)$. In this situation,
\begin{equation*}
\begin{split}
\mc L_n f(\eta)  = & \Big(\frac{\alpha}{2n^\beta}+\frac{a}{2n^\gamma}\Big)\eta(-1)(1-\eta(0)) -  \Big(\frac{\alpha}{2n^\beta}-\frac{a}{2n^\gamma}\Big)\eta(0)(1-\eta(-1))\\
& +  \Big(\frac{1}{2}-\frac{a}{2n^\gamma}\Big)\eta(1)(1-\eta(0)) -  \Big(\frac{1}{2}+\frac{a}{2n^\gamma}\Big)\eta(0)(1-\eta(1))\,,\\
\end{split}
\end{equation*}
leading to \eqref{eq2.3}.
To check the remaining cases is a long albeit simple calculation and we leave it to the reader.

\bigskip

We denote by $\chi(\rho)$ the static compressibility of the system defined as $\chi(\rho):=\rho(1-\rho)$. In the following we consider the Markov process starting from the equilibrium measure $\nu_\rho$. For the sake of clarity we denote the probability measure $\bb P_{\nu_\rho}$ by $\bb P_\rho$ and the corresponding expectation by $\bb E_\rho$.

\section{Statement of the main results}
\label{sec:results}
\subsection{Boltzmann-Gibbs principle}
The second order estimate below plays an essential  role in the proof of density fluctuations.
Based on \cite{gj2014}, we give a new version of the second-order Boltzmann-Gibbs principle. Roughly speaking, we look at the first-order correction for the usual limit projection of the space-time fluctuations of some specific field. More precisely, we focus on the functional $(\eta(x)-\rho)(\eta(x+1)-\rho)$ and show how its fluctuations can be written as a linear functional of the density field plus a quadratic functional of this same field. The crucial point relies on the sharp quantitative bounds on the error that we are able to obtain when we perform the aforementioned replacement.

Hereafter we denote $\bar{\eta}(y)=\eta(y)-\rho$ the centered occupation variable at the site $y\in \bb Z$. In the following, we let $C$  denote  a constant that does not depend on the variables $ L, n$ nor $t$ introduced below.

\begin{theorem}[Second-order Boltzmann-Gibbs principle]\label{theo:BG}

Let $v:\mathbb{Z}\to{\mathbb{R}}$ be a  function such that
\begin{equation}\label{vsummable}
 \|v\|_{2,n}^2:=\frac{1}{n}\sum_{x\in\mathbb{Z}}v^2(x)<\infty.
\end{equation}
There exists a constant $C>0$, such that for any $t>0$, and any positive integers $L, n$:
\begin{multline}
\label{eq:BGexpo}
\mathbb{E}_{{\rho}}\bigg[\Big(\int_{0}^t \sum_{x\in\mathbb{Z}}v(x)\Big\{\bar{\eta}_{sn^{2}}(x)\bar{\eta}_{sn^{2}}(x+1)-\big(\vec{\eta}_{sn^{2}}^{L}(x)\big)^2+\frac{\chi(\rho)}{L}\Big\}\; ds\Big)^2\bigg]\\
\leq  Ct\Big\{\frac{L}{n} +\frac{n^\beta}{\alpha n}+\frac{tn}{L^2}\Big\}\|v\|_{2,n}^2 + Ct\Big\{\frac{n^\beta(\log_2(L))^2}{\alpha n}\Big\}\frac{1}{n}\sum_{x\neq -1} v^2(x),
\end{multline}
where $\vec{\eta}^{L}(x)$ denotes the empirical average on the box of size $L$ situated at the right of site $x$, and is defined by
\begin{equation}
\vec{\eta}^{L}(x)=\frac{1}{L}\sum_{y=x+1}^{x+L}\bar{\eta}(y). \label{eq:mean}
\end{equation}
\end{theorem}

We report the proof of Theorem \ref{theo:BG}  to Section \ref{sec:BG}.
At the macroscopic level, the density fluctuation field turns out to be of two types (depending on the values of the parameters): either an OU process driven by the heat equation with some boundary conditions, or an energy solution of the SBE equation (as it has been introduced in \cite{gj2014}). The next two subsections give the corresponding definitions.

\subsection{The Ornstein-Uhlenbeck process}
Following \cite{fgn3}, we give here a characterization of the generalized OU process $\mc Y^{\beta}_\bola$ which is a solution of
\begin{equation}
\label{ou}
d\mc Y^{\beta}_t = \frac{1}{2} \Delta_\beta \mc Y^{\beta}_t dt + \sqrt{\chi(\rho)} \nabla_\beta d \mc W_t^{\beta},
\end{equation}
in terms of a martingale problem, where $\{\mc W^{\beta}_t \; ; \;t \in [0,T]\}$ is a $\mc S_\beta'(\bb R)$-valued Brownian motion.

From here on, let $\mc D([0,T],\mc S_\beta'(\bb R))$ (resp.  $\mc C([0,T],\mc S_\beta'(\bb R))$) denote the set of trajectories which are right-continuous and have left limits (resp. continuous) taking values in $\mc S_\beta'(\bb R)$.

\begin{proposition}{\cite[Section 5]{fgn3}}\label{prop:sol_ou}
There exists a unique field $\mc Y^{\beta}_\bola$ which takes values in $\mc C([0,T],\mc S_\beta'(\bb R))$ such that
\begin{enumerate}[(i)]
\item for every $\vphi \in \mc S_\beta(\bb R)$,
\begin{align*}\mathfrak{M}^{\beta}_t(\vphi) & :=\mc Y^{\beta}_t(\vphi) - \mc Y^{\beta}_0(\vphi) - \frac{1}{2}\int_0^t \mc Y^{\beta}_s(\Delta_\beta \vphi)\; ds \\
\mathfrak{N}^{\beta}_t(\vphi)& := \big(\mathfrak{M}^{\beta}_t(\vphi)\big)^2-{\chi(\rho)}t\big\|\nabla_\beta \vphi \big\|_{2,\beta}^2
\end{align*}
are martingales with respect to the filtration $\mc F_t:=\sigma\big(\mc Y^{\beta}_s(\vphi)\; ;\; s \le t, \vphi \in \mc S_\beta(\bb R)\big)$;
\item the field $\mc Y^{\beta}_0$ is a Gaussian field of mean zero and covariance given on $\vphi,\psi \in \mc S_\beta(\bb R)$ by
\[
\bb E_\rho\big[\mc Y^{\beta}_0(\vphi)\mc Y^{\beta}_0(\psi)\big]=\chi(\rho) \int_{\bb R} \vphi(u)\psi(u) du.
\]
Moreover, for each $\vphi \in \mc S_\beta(\bb R)$, the stochastic process $\{\mc Y^{\beta}_t(\vphi) \; ;\; t  \in [0,T]\}$ is Gaussian, being the distribution of $\mc Y^{\beta}_t(\vphi)$ conditionally to $\{\mc F_u \; ; \; u\leq s\}$, normal of mean $\mc Y^{\beta}_s(T_{t-s}^\beta \vphi)$ and variance \begin{equation}
\int_0^{t-s} \big\|\nabla_\beta T_r^\beta \vphi\big\|_{2,\beta}^2 \; dr,\label{eq:variance}
\end{equation} where $T_t^\beta$ is the semi-group associated to $\Delta_\beta$ as defined in Appendix \ref{app}.
\end{enumerate}
\end{proposition}
We call the random field $\mc Y^{\beta}_\bola$ the generalized OU process of characteristics $\nabla_\beta$ and $\Delta_\beta$. The proof of Proposition \ref{prop:sol_ou} is completely similar to the proof given in \cite[Section 5]{fgn3}.

\subsection{The stochastic Burgers equation}
In this section we recall from \cite{gj2014} the notion of \emph{stationary energy solutions} for the SBE, which reads as
\begin{equation}
\label{eq:sbe1}
d\mc Y_t=\frac{1}{2}  \Delta \mc Y_t dt + a \nabla \mc Y^2_t dt + \sqrt{\chi(\rho)} \nabla d\mc W_t ,
\end{equation}
where $\{\mc W_t\; ; \; t \in [0,T]\}$ is a $\mc S'(\bb R)$-valued Brownian motion. For $\varepsilon >0$ and $x \in \bb R$, we define $i_\varepsilon(x) : \bb R \to \bb R$ as $i_\varepsilon(x)(y)=\varepsilon^{-1}\mathbf{1}_{x<y\le x+\varepsilon}$ for any $y \in \bb R$.

\begin{definition}\label{def:energysol}
A stochastic process $\{\mc Y_t\; ; \;t \in [0,T]\}$ is a stationary energy solution of the SBE \eqref{eq:sbe1} if
\begin{enumerate}
\item the process $\{\mc Y_t \;;\; t \in [0,T]\}$ is stationary, namely, for each $t \in [0,T]$, the  $\mc S'(\bb R)$ random variable $\mc Y_t$ is a white noise of variance $\chi(\rho)$, and it satisfies the following energy estimate: there exists a positive constant $\kappa$ such that
\begin{enumerate}
\item for any $\vphi \in \mc S(\bb R)$ and any $ s,t \in [0,T]$, $s<t$,
\begin{equation}\label{eq:ec1}
\bb E_\rho \bigg[ \Big(\int_s^t \mc Y_r( \Delta \vphi) dr\Big)^2\bigg] \le \kappa (t-s) \|\nabla \vphi\|_2^2.
\end{equation}
\item for any $\vphi \in \mc S(\bb R)$, any $\delta,\varepsilon \in (0,1), \delta < \varepsilon$, and any $s, t \in [0,T]$, $s<t$,
\begin{equation}\label{eq:ec2}
\bb E_\rho\Big[ \big(\mc A_{s,t}^\varepsilon(\vphi)-\mc A_{s,t}^\delta(\vphi)\big)^2\Big] \le \kappa (t-s) \varepsilon \|\nabla \vphi\|_2^2,
\end{equation}
where \[
\mc A_{s,t}^\varepsilon(\vphi)=\int_s^t\int_{\bb R} \mc Y_r\big(i_\varepsilon(x)\big)^2\nabla\vphi(x)dxdr.
\]
\end{enumerate}
\item for any test function $\vphi \in \mc S(\bb R)$ and any $t \in [0,T]$, the process
\[ \mc Y_t(\vphi)-\mc Y_0(\vphi)-\frac{1}{2}\int_0^t \mc Y_s(\Delta \vphi)\; ds - a \mc A_t(\vphi)\]
is a continuous martingale of quadratic variation $t \chi(\rho)  \|\nabla \vphi\|_2^2$, where the process $\{\mc A_t \;; \;t\in [0,T]\}$ is obtained through the following limit, which holds  in the $\mathbf{L}^2(\bb P_\rho)$-sense:
\begin{equation} \mc A_t(\vphi)-\mc A_s(\vphi):=\lim_{\varepsilon \to 0} \mc A_{s,t}^\varepsilon(\vphi).\label{eq:processA}
\end{equation}
\end{enumerate}
\end{definition}

The proof of the existence of $\{\mc A_t \; ;\; t\in [0,T]\}$ given by \eqref{eq:processA} is completely done in \cite[Theorem 1]{gj2014}.
We recall here the result:
\begin{theorem}[\cite{gj2014}]
Let $\{\mc Y_t \; ; \; t \in [0,T]\}$ be a stationary process such that \eqref{eq:ec2} is satisfied. Then, there exists a $\mc S'(\bb R)$-valued stochastic process $\{\mc A_t \; ; \; t \in [0,T]\}$ with continuous paths such that
\[\mc A_t(\vphi)=\lim_{\varepsilon \to 0} \int_0^t\int_{\bb R} \mc Y_r\big(i_\varepsilon(x)\big)^2\nabla\vphi(x)dxdr \] is in $\mathbf{L}^2(\bb P_\rho)$, for any $t\in [0,T]$ and any $\vphi \in \mc{S}(\bb R)$.

\end{theorem}

\begin{remark} Let us notice that the function $i_\varepsilon(x)$ does not belong to the space of test functions of  $\mc Y_r$. Nevertheless, by doing a similar argument to the one of \cite[Lemma 4.6]{fgn2014} (see also \cite[Remark 4]{gj2014}), one can show that $\mc Y_r(i_\varepsilon(x))$ is well defined. \end{remark}

We are now ready to state the main theorems.

\subsection{The density fluctuation field}
Fix $\beta\geq 0, \gamma>0$ and $n\in\bb N$. The {\em density fluctuation field} $\{\mc Y_t^{\beta,\gamma,n}\;;\; t \in [0,T]\}$ is the $\mc S'_\beta(\bb R)$-valued process given on $\vphi \in \mc S_\beta(\bb R)$ by
\[
\mc Y_t^{\beta,\gamma,n}(\vphi): = \frac{1}{\sqrt n} \sum_{x \in \bb Z} (\eta_{tn^2}(x)-\rho) \vphi\Big(\frac{x}{n}\Big).
\]

When there is no asymmetry (which corresponds to $a=0$), we already know from \cite{fgn3} the asymptotic behavior of the density fluctuation field. More precisely:
\begin{theorem}[\cite{fgn3}]
If $a=0$, the sequence of processes $\{\mc Y_t^{\beta,\gamma,n}\;;\; t \in [0,T]\}_{n \in \bb N}$ converges in distribution with respect to the Skorokhod topology of $\mc D([0,T];\mc S'_\beta(\bb R))$, as $n \to\infty$, to the  OU process given by the solution of \eqref{ou} (in the sense of Proposition \ref{prop:sol_ou}).
\end{theorem}

Here we are interested in the case $a\neq 0$.
We redefine
\begin{equation}
\mc Y^{\beta,\gamma,n}_t(\vphi) := \frac{1}{\sqrt n} \sum_{x \in \bb Z} (\eta_{tn^2}(x)-\rho) \vphi\bigg( \frac{x-n^{2-\gamma}a(1-2\rho)t}{n}\bigg)
\label{eq:field}
\end{equation}
for any $\vphi \in \mc S_\beta(\bb R)$. For the sake of clarity, from now on, we denote $\mc Y^{\beta,\gamma,n}_t$ by $\mc Y^{n}_t$, but we keep in mind that the fluctuation field (as well as its limit) strongly depends on $\beta$ and $\gamma$.
For the reader convenience, let us assume that $\rho=1/2$. We notice however that all the results hold for any $\rho$.
\begin{theorem}[Ornstein-Uhlenbeck process]\label{thm:crossover}
If one of these two conditions are satisfied:
\begin{itemize}
\item $\beta  \leq 1/2$ and $\gamma > 1/2$,
\item $\gamma \geq \beta > 1/2$, 
\end{itemize}then the sequence of processes $\{\mc Y^{n}_t\;;\; t \in [0,T]\}_{n \in \bb N}$ converges in distribution with respect to the Skorokhod topology of $\mc D([0,T]\;;\;\mc S'_\beta(\bb R))$, as $n\to\infty$, to the  OU process given by \eqref{ou}, in the sense of Proposition \ref{prop:sol_ou}.
\end{theorem}

In the regime described above, the asymmetry has no effect in the limit, since we recover (in particular) the  results of  \cite{fgn3}. However, for the special case $\beta=\gamma=1$,  the symmetry and the asymmetry scale with the same strength at the slow bond, and this particularity translates into a different strength for the noise. In particular, the variance given in \eqref{eq:variance} has an extra term which depends on $a$ (the asymmetry parameter), see Definition \ref{def:norm}.

On the other hand, when $\gamma=1/2$, the asymmetry is quite strong and  gives rise to the non-linear term in the SBE:
\begin{theorem}[Stochastic Burgers equation] \label{theo:fluct}
Assume  $\gamma=1/2$. For any $\beta\le 1/2$,  the sequence of processes $\{\mc Y^{n}_t\;;\; t \in [0,T]\}_{n \in \bb N}$ is tight with respect to the Skorokhod topology of   $\mc D([0,T]\;;\; \mc S'_\beta(\bb R))$ and  any limit point is a stationary energy solution of the stochastic Burgers equation \eqref{eq:sbe1}.
\end{theorem}

In other words, for $\beta<1$, in which case $\mc S_\beta(\bb R)$ is the classical Schwartz space $\mc S(\bb R)$, we recover two known results: \begin{itemize}
\item when we consider the process without the slow bond (i.~e.~the classical SSEP, which corresponds to $\alpha=1, \beta=0$, $a=0$): see \cite{Ravi1992};
\item when we consider the process without the slow bond but with a weak asymmetry (which corresponds to $\alpha=1$, $\beta=0$ and $\gamma \in(1/2,1]$): see \cite{gj2014}.
\end{itemize}

\section{Elements of proof}
\label{sec:comp}

The strategy of the proof consists in analysing the following martingale:
\begin{equation}
 \mc M_t^{n}(\vphi):= \mc Y_t^{n}(\vphi)-\mc Y_0^{n}(\vphi)-\int_0^t n^2 \mc L_n \big\{\mc Y_s^{n}(\vphi)\big\} ds,
 \label{eq:martin}
\end{equation}
for any $\vphi \in \mc S_\beta(\bb R)$. The fact that $\{\mc M_t^n(\vphi)\; ;\; t\geq 0\}$ is a martingale is a consequence of Dynkin's formula, see \cite[Lemma A.1.5.1]{kl}. In this section we give some properties of $\mc M_t^n$, after rewriting it in a suitable way.
\subsection{Microscopic current}
Since the dynamics conserves the total number of particles, for any $x\in\mathbb Z$ there exist instantaneous currents $j_{x,x+1}^n$ such that \[n^2 \mathcal{L}_n(\eta(x))=j_{x-1,x}^n(\eta)-j_{x,x+1}^n(\eta).\]
Straightforward computations show that the instantaneous current reads
\begin{equation*}
j_{x,x+1}^n(\eta)= j_{x,x+1}^{n,S}(\eta) + j_{x,x+1}^{n,A}(\eta)\end{equation*}
where
\begin{align*}
j_{x,x+1}^{n,A}(\eta) & = \frac{an^2}{2n^\gamma}(\eta(x+1)-\eta(x))^2,\ \qquad x \in \bb Z,\\
j_{x,x+1}^{n,S}(\eta) & = \frac{n^2}{2}(\eta(x)-\eta(x+1)), \qquad \quad x\neq -1,\\
j_{-1,0}^{n,S}(\eta) & =  \frac{\alpha n^2}{2n^\beta}(\eta(-1)-\eta(0)).
\end{align*}
For any $x \in \bb Z$, the expectation of the instantaneous current $j^n_{x,x+1}$ with respect to the measure $\nu_\rho$ is equal to $an^{2-\gamma} \chi(\rho)$. In other words, some transport behavior appears in the density fluctuations. In order to see a non-trivial evolution, we need to recenter the density fluctuation field as in \eqref{eq:field}, except if $\rho=1/2$, meaning that the transport velocity vanishes.


\subsection{Martingale decomposition}

Let us introduce the operator $\bb L_n:=\bb L^A_n + \bb L^S_n$, where $\bb L_n^A$ and $\bb L_n^S$ are acting on functions $\vphi \in \mc S_\beta(\bb R)$ as 
\begin{align*}
\bb L_n^A \vphi\Big(\frac{x}{n}\Big) & := \frac{a\sqrt n}{2 n^{\gamma}} \Big\{\vphi\Big(\frac{x+1}{n}\Big) - \vphi\Big(\frac{x}{n}\Big)\Big\}\\
\bb L_n^S \vphi\Big(\frac{x}{n}\Big)&:= \zeta_{x,x+1}^n \Big\{\vphi\Big(\frac{x+1}{n}\Big) - \vphi\Big(\frac{x}{n}\Big)\Big\} +   \zeta_{x-1,x}^n \Big\{\vphi\Big(\frac{x-1}{n}\Big) - \vphi\Big(\frac{x}{n}\Big)\Big\}
\end{align*}
where \[\zeta_{x,x+1}^n:=\begin{cases} 1/2\; ; & x \neq -1 \\ \alpha/(2n^\beta)\; ; &  x=-1. \end{cases} \]
 After simple computations we can rewrite the last term of \eqref{eq:martin} as
\begin{multline*}
\int_0^t n^2 \mc L_n \big\{\mc Y_s^{n}(\vphi)\big\} ds=\int_0^t \frac{1}{\sqrt n} \sum_{x \in \bb Z} n^2 \bb L_n^S\vphi\Big(\frac{x}{n}\Big) \eta_{sn^2}(x) \; ds \\ + \int_0^t \frac{1}{\sqrt n} \sum_{x \in \bb Z} n^2 \bb L_n^A\vphi\Big(\frac{x}{n}\Big) (\eta_{sn^2}(x)-\eta_{sn^2}(x+1))^2 \; ds \end{multline*}
and by the definition of $\bb L_n^A$ and $\bb L_n^S$ we have that
\begin{equation*}
 \mc M_t^{n}(\vphi):= \mc Y_t^{n}(\vphi)-\mc Y_0^{n}(\vphi)- \mc I_t^{n}(\vphi)-\mc R_t^{n}(\vphi)-\mc B_t^{n}(\vphi),
\end{equation*}
where
\begin{align*}
\mc I_t^{n}(\vphi) &:= \frac{1}{2}\int_0^t \mc Y_s^{n}(\Delta_\beta \varphi) \;ds=\frac{1}{2}\int_0^t \frac{1}{\sqrt n}\sum_{x \in \bb Z} (\eta_{sn^2}(x)-\rho) \Delta_\beta\varphi\Big( \frac x n \Big) \; ds,\\
\mc R_t^{n}(\vphi)&:= \int_0^t \frac{1}{\sqrt n} \sum_{x \in \bb Z} n^2 \bb L_n^S\vphi\Big(\frac{x}{n}\Big) \eta_{sn^2}(x) \; ds- \mc I_t^{n}(\vphi),
\end{align*}
and since $\sum_{x\in\bb Z}\nabla_x^n\vphi=0$, the last term $\mc B_t^n(\vphi)$ can be written as
\[\mc B_t^{n}(\vphi):= \int_0^t \sum_{x \in \bb Z} \tau_xF_n(\eta_{sn^2})
\nabla_x^n\vphi\;ds,
\]
where $\tau_x$ denotes the translation operator that acts on a function $h:\Omega\to\bb R$ as $(\tau_x h)(\eta):=h(\tau_x \eta)$, and $\tau_x \eta$ is the configuration obtained from $\eta$ by shifting: $(\tau_x \eta)_y=\eta_{x+y}$ and with $F_n:\Omega \to \bb R$  defined by \begin{equation*}
F_n(\eta):= \frac{a\sqrt n}{2n^\gamma} \Big\{(\eta(1)-\eta(0))^2-2\chi(\rho)\Big\}.\end{equation*}
Let us remark that in the case $\rho=1/2$ we have
\begin{equation*}
(\eta(1)-\eta(0))^2-2\chi(\rho)=-2(\eta(1)-\rho)(\eta(0)-\rho).
\end{equation*}
Recalling the notation for $x \in \bb Z$, $\bar{\eta}(x)=\eta(x)-\rho$, we get
\[
F_n(\eta)=\frac{ -a \sqrt n}{n^\gamma} \; \bar{\eta}(1) \bar{\eta}(0).
\]

\begin{remark} Notice that the field $\mc I_t^{n}$ comes from the symmetric part of the current therefore it does not depend on $a$ nor $\gamma$; while $\mc B_t^{ n}$ comes from the asymmetric part and it depends on $a$, $\beta$ and $\gamma$ (the dependence on $\beta$ is hidden in the boundary condition of the test function $\vphi$).
\end{remark}

\subsection{Effects of the slow bond}


We first observe that for any $\beta \geq 0$, $\mc R_t^{n}(\vphi)$ is negligible in $\mb L^2(\bb P_\rho)$ for every $\vphi \in \mc S_\beta(\bb R)$. Since $\mc R_t^{n}$ does not depend on $a$ nor $\gamma$, the proof is similar to the one given in \cite{fgn3} for Proposition 3.1. We notice however that the model in \cite{fgn3} corresponds in our case to the choice $a=0$ and  \cite[Proposition 3.1]{fgn3} is a consequence of  \cite[Lemma 7.2]{fgn3}, which can be derived in the same way for our model by noticing that the Dirichlet form defined in \eqref{eq:dirichlet} is greater or equal than the Dirichlet form of the model in \cite{fgn3}. We refer the reader  to that paper, and we only state the result:
\begin{proposition}(\cite{fgn3})\label{vanishingofR}
Let us consider $t \in [0,T]$ and $\vphi \in \mc S_\beta(\bb R)$. Then,
\[ \lim_{n \to \infty} \bb E_\rho\left[ \big(\mc R_t^{n}(\vphi)\big)^2\right]=0.\]
\end{proposition}
To sum up, we can rewrite the martingale, for $n$ sufficiently large,  as
\begin{equation}
\mc M_t^{n}(\vphi):= \mc Y_t^{n}(\vphi)-\mc Y_0^{n}(\vphi)-\frac{1}{2}\int_0^t \mc Y_s^{n}(\Delta_\beta \vphi) \; ds - \mc B_t^{n}(\vphi) +o(1) ,\label{eq:decomposition}
\end{equation}
where $o(1)$ vanishes in $\bf L^2(\bb P_\rho)$, as $n\to\infty$. 
There exists a range of $(\gamma,\beta)$ for which the contribution of the antisymmetric part of the current is negligible, as this can be seen through the following proposition:

\begin{proposition}\label{vanishingofB}
If one of these two conditions are satisfied:
\begin{itemize}
\item $\beta  \leq 1/2$ and $\gamma > 1/2$,
\item $\gamma \geq \beta > 1/2$,
\end{itemize}  we have, for any $t \in [0,T]$ and for any $\vphi \in \mc S_\beta(\bb R)$,
\[ \lim_{n \to \infty} \bb E_\rho\left[ \big(\mc B_t^{n}(\vphi)\big)^2\right]=0.\]
\end{proposition}
The proof of Proposition \ref{vanishingofB}  is given in Section \ref{vanishB}.

\subsection{Quadratic variation}
The quadratic variation of the martingale $\mc M_t^{n}(\vphi)$ is
\begin{equation}
\label{eq:quadratic}
\langle \mc M^{n}(\vphi) \rangle_t= \int_0^t \frac{1}{2n} \sum_{x\neq -1} \tau_x G_n(\eta_{sn^2}) \big(\nabla_x^n \vphi\big)^2 \;ds + \int_0^t \frac{1}{2n} H_n(\eta_{sn^2})  \big(\nabla_{-1}^n \vphi\big)^2\; ds  ,
\end{equation}
where $G_n, H_n:\Omega \to \bb R$ are defined by
\begin{align*}
G_n(\eta) & := (\eta(1)-\eta(0))^2+\frac{a}{n^\gamma} \eta(1)(1-\eta(0)),\\
H_n(\eta)& :=  \frac{\alpha}{n^\beta}(\eta(-1)-\eta(0))^2+\frac{a}{n^\gamma} \eta(0)(1-\eta(-1)).
\end{align*}

\begin{lemma}\label{lem:quadratic-var}
For any $\gamma\geq\beta\geq 0$, and $\vphi \in \mc S_\beta(\bb R)$, the quadratic variation $\langle \mc M^{n}(\vphi) \rangle_t$ converges in $\mathbf L^2(\bb P_\rho)$, as $n \to\infty$, towards $t\chi(\rho) \|\nabla_\beta\vphi\|_{2,\beta}^2$.
\end{lemma}

\begin{proof}
According to \eqref{eq:quadratic}, we write  $\bb E_\rho\big[\langle \mc M^{n}(\vphi) \rangle_t\big]$ as the sum below of  four terms:
\begin{align}
& \int_0^t \frac{1}{2n} \sum_{x \neq -1} \bb E_\rho\Big[ (\eta_{sn^2}(x+1)-\eta_{sn^2}(x))^2 \Big]  \big(\nabla_x^n \vphi\big)^2 ds \label{eq:q1}\\
& + \int_0^t \frac{a}{2n^{1+\gamma}}  \sum_{x \neq -1} \bb E_\rho\Big[\eta_{sn^2}(x+1) (1-\eta_{sn^2}(x)) \Big] \big(\nabla_x^n \vphi\big)^2 ds \label{eq:q2}\\
& + \int_0^t \frac{\alpha}{2n^{1+\beta}}\bb E_\rho\Big[ (\eta_{sn^2}(0)-\eta_{sn^2}(-1))^2 \Big] n^2 \Big\{ \vphi\Big(\frac{0}{n}\Big)-\vphi\Big(\frac{-1}{n}\Big)\Big\}^2 ds \label{eq:q3}\\
& + \int_0^t \frac{a}{2n^{1+\gamma}}\bb E_\rho\Big[ \eta_{sn^2}(0)(1-\eta_{sn^2}(-1)) \Big] n^2 \Big\{ \vphi\Big(\frac{0}{n}\Big)-\vphi\Big(\frac{-1}{n}\Big)\Big\}^2 ds. \label{eq:q4}
\end{align}
First, we notice that the limits of \eqref{eq:q1} and \eqref{eq:q2} do not depend on $\beta$, since the sum avoids the slow bond, and for any regime of $\beta$ the function $\varphi$ which is involved in the sum is smooth. More precisely:
\begin{itemize}
\item \eqref{eq:q1} converges to $ t \chi(\rho) \|\nabla_\beta\vphi\|_2^2$ as $n\to\infty$;
\item if $\gamma> 0$, \eqref{eq:q2} converges to $0$ as $n\to\infty$.
\end{itemize}
Now we divide the proof into three cases depending on the range of $\beta$.
\begin{enumerate}
\item \textbf{Case $\beta < 1$.}
In this case, the test function $\vphi$ belongs to $\mc S(\bb R)$. Therefore, \begin{itemize}
\item \eqref{eq:q3} is of order $O_\vphi(n^{-(1+\beta)})$ and  \eqref{eq:q4} is of order $O_\vphi(n^{-(1+\gamma)})$, so they both vanish as $n\to\infty$.
\end{itemize}

\item \textbf{Case $\beta = 1$.}
In this case, $\varphi$ satisfies the boundary condition \eqref{eq:bound_beta_1}. Then one can easily see that
\begin{itemize}
\item \eqref{eq:q3} converges as $n\to\infty$ to
\[ t\chi(\rho) \alpha \big(\vphi(0^+)-\vphi(0^-)\big)^2 = \frac{t\chi(\rho)}{\alpha} \Big( \frac{d\vphi}{du}(0^+)\Big)^2.\]
\item if $\gamma >1$, \eqref{eq:q4} vanishes as $n\to\infty$;
\item if $\gamma=1$, \eqref{eq:q4} converges as $n\to\infty$ to
\[ t\chi(\rho) a \big(\vphi(0^+)-\vphi(0^-)\big)^2 = \frac{t\chi(\rho)a}{\alpha^2} \Big( \frac{d\vphi}{du}(0^+)\Big)^2.\]
\end{itemize}

\item \textbf{Case $\beta > 1$.}
In this case, for any $\gamma \ge \beta$, \eqref{eq:q3} and \eqref{eq:q4} vanish as $n\to\infty$.
\end{enumerate}
\end{proof}

\section{Proof of Theorem \ref{thm:crossover} and Theorem \ref{theo:fluct} }
\label{sec:proof}

The main ingredient for proving both theorems is the second order Boltzmann-Gibbs principle stated in Theorem \ref{theo:BG}, whose proof is postponed to Section \ref{sec:BG}, which gives rise to the SBE (Theorem \ref{theo:fluct}). 

We start by showing the zero contribution of $\mc B_t^n$ in some particular range of $(\beta,\gamma)$.In what follows, we will write $\varepsilon n$, resp. $cn$, for $\lfloor \varepsilon n\rfloor$, resp. $\lfloor cn\rfloor$, its integer part.

\subsection{Proof of Proposition \ref{vanishingofB}}\label{vanishB}
Recall that $ \mc B_t^{n}(\vphi)$ can be written as
\begin{equation*}
\mc B_t^{n}(\vphi)=-a\frac{\sqrt n}{n^\gamma}\int_0^t \sum_{x \in\bb Z} \bar\eta_{sn^2}(x+1) \bar\eta_{sn^2}(x)\nabla_x^n\vphi \;ds,
\end{equation*}
where, for $x\in\mathbb{Z}$, $\bar{\eta}(x)=\eta(x)-\rho$.
 
From the convexity inequality $(x+y+z)^2\le 3(x^2+y^2+z^2)$, we can bound the 
$\bf L^2(\bb {P}_{\rho})$-norm of $\mc B_t^{n}(\vphi)$  by 
\begin{align}
& 3\frac{a^2\; n}{n^{2\gamma}}\;\mathbb{E}_{\rho} \bigg[ \Big(  \int_0^t\sum_{x \neq -1} \nabla_x^n\varphi\Big\{\bar\eta_{sn^2}(x+1) \bar\eta_{sn^2}(x)-\big(\vec\eta_{sn^2}^L(x)\big)^2+\frac{\chi(\rho)}{L}\Big\}\, ds\Big)^2\bigg]\label{eq:newterm1}\\
&+ 3 \frac{a^2\; n}{n^{2\gamma}}\;\mathbb{E}_{\rho} \bigg[ \Big(  \int_0^t\sum_{x \neq -1} \nabla_x^n\varphi \Big\{\big(\vec\eta_{sn^2}^L(x)\big)^2-\frac{\chi(\rho)}{L}\Big\}\, ds\Big)^2\bigg]\label{eq:newterm2}
\\&+3\frac{a^2\; n}{n^{2\gamma}}(\nabla_{-1}^n \varphi)^2\;\mathbb{E}_{\rho} \bigg[ \Big( \int_0^t   \bar\eta_{sn^2}(-1) \bar\eta_{sn^2}(0) \; ds \Big)^2\bigg]. \label{eq:newterm3}
\end{align}
By Theorem \ref{theo:BG} the first term \eqref{eq:newterm1} is bounded by
\begin{equation}Ct\;\frac{a^2\;n}{n^{2\gamma}}\;\Big\{\frac{L}{n} +\frac{n^\beta}{\alpha n}+\frac{tn}{L^2}+\frac{n^\beta(\log_2(L))^2}{\alpha n}\Big\} \times \frac{1}{n}\sum_{x\neq -1}(\nabla_x^n\varphi)^2.\label{eq:BD1}\end{equation}
By independence and by the Cauchy-Schwarz inequality the second term \eqref{eq:newterm2} is bounded by
\begin{multline}
C t^2  \; \frac{a^2\;n}{n^{2\gamma}}\; \Big\{ L\sum_{x\neq -1}(\nabla_x^n\varphi)^2\Big\}\int \Big(\big(\vec{\eta}^L(0)\big)^2-
\frac{\chi(\rho)}{L}\Big)^2 \nu_\rho(d\eta)\\ \le Ct^2\; \frac{a^2\;n}{n^{2\gamma}}\times \frac{n}{L}\; \Big\{ \frac{1}{n}\sum_{x\neq -1}(\nabla_x^n\varphi)^2\Big\}. \label{eq:BD2}
\end{multline}
The last term \eqref{eq:newterm3} is estimated by using Proposition \ref{prop:estim-1}, which is stated and proved in Section \ref{sec:auxi}. This term is the most delicate one because it depends on the continuity of the test function at $0$: indeed, when $\varphi\in\mc S_\beta(\bb R)$, the quantity $\nabla_{-1}^n\varphi$   can be of order  $1$ if $\beta<1$, or  of order $n$ if $\beta\geq 1$.
 For that purpose we  distinguish now two cases:\begin{enumerate}
\item If $\beta < 1$ then the test function $\varphi$ is in the Schwartz space $\mc S(\bb R)$, therefore
\[
(\nabla_{-1}^n\varphi)^2 \xrightarrow[n\to\infty]{} \Big\{\frac{d\varphi}{du}(0)\Big\}^2,
\]
and 
\begin{equation}\label{eq:bound-int}
\sup_{n\geq 0} \bigg\{ \frac{1}{n} \sum_{x\neq -1} \big(\nabla_x^n\varphi\big)^2\bigg\} < +\infty,
\end{equation}
hence,  from Proposition \ref{prop:estim-1} and the bounds \eqref{eq:BD1} and \eqref{eq:BD2} above, we have: for any $L\in\mathbb N $ and $\epsilon>0$, $\mathbb{E}_{\rho} [ (\mc B_t^{n}(\vphi))^2]$ can be bounded by 
\begin{equation}\label{eq:boundB}
C(t,a,\varphi)\frac{n}{n^{2\gamma}}\;\Big\{\frac{L}{n} +\frac{n^\beta}{\alpha n}+\frac{tn}{L^2}+\frac{n^\beta(\log_2(L))^2}{\alpha n}+\frac{tn}{L} +\frac{1}{n^{1+\epsilon}}
\Big\}\end{equation}
which vanishes, as $n\to \infty$, for any $\gamma > 1/2$, after taking $L=cn$, with $c>0$ being a positive constant.
\item If $\beta \geq 1$, then the test function $\varphi$ has a discontinuity at $0$, and we have
\[
  \Big\{\varphi\Big(\frac{-1}{n}\Big)-\varphi\Big(\frac{0}{n}\Big)\Big\}^2 \xrightarrow[n\to\infty]{} \big\{\varphi(0^+)-\varphi(0^-)\big\}^2.
\]
Notice  however that \eqref{eq:bound-int} remains valid,  and we   deduce that, for any $L\in\mathbb N $ and $\epsilon>0$,  $\mathbb{E}_{\rho} [ (\mc B_t^{n}(\vphi))^2]$ can be bounded by 
\begin{equation} \label{eq:errorfinal}
C(t,a,\varphi)\frac{n}{n^{2\gamma}}\; \Big\{\frac{L}{n} +\frac{n^\beta}{\alpha n}+\frac{tn}{L^2}+\frac{n^\beta(\log_2(L))^2}{\alpha n}+\frac{tn}{L} +n^{1-\epsilon}\Big\}.\end{equation}
Notice that the term $n^{1-\varepsilon}$  comes from Proposition \ref{prop:estim-1} plus the fact that $(\nabla_{-1}^n\varphi)^2$ is of order $n^2.$
If one lets $n\to \infty$, and then $L\to\infty$, then one can see that \eqref{eq:errorfinal} vanishes for every $\gamma \geq \beta\geq 1$, which concludes the proof of Proposition \ref{vanishingofB}.
\end{enumerate}

 Let us now follow the lines of \cite{gj2014}: first, for any value of $\beta$ we prove that the sequence of processes $\{\mc Y_t^{n}\; ;\; t \in [0,T]\}_{n \in \bb N}$ is tight with respect to the Skorokhod topology of $\mc D([0,T],\mc S'_\beta(\bb R))$. Then, we prove that \begin{itemize}
\item if $\beta \le 1/2$ and $\gamma=1/2$, then any limit point of that sequence is an energy solution of the SBE \eqref{eq:sbe1};
\item in the other cases (whenever \eqref{eq:hyp} is satisfied), any limit point of that sequence solves the martingale problem given in Proposition \ref{prop:sol_ou}.
\end{itemize}

\subsection{Tightness of the density field}

We prove tightness of the sequence of processes $\{\mc Y_t^{n}\; ;\; t \in [0,T]\}_{n \in \bb N}$.
The proof relies on three well-known criteria. The first one is due to Mitoma \cite{Mitoma} and reduces the proof of tightness for distribution-valued processes to the proof of tightness for real-valued processes. We notice that we can use Mitoma's criterion, since for all the regimes of $\beta$, the space $\mc S_\beta(\bb R)$ is a Fr\'echet space (see \cite{fgn3}).  The second criterion is due to Aldous (see \cite{gj2014,kl}  for example) and allows us to work with the Skorokhod topology, whereas the last one, due to Prohorov, Kolmogorov and Centsov \cite{Karatzas}, treats the case of processes with continuous paths. The proofs are similar to those of \cite{gj2014}. We recall here the main steps for the sake of completeness, and underline the new role played by the slow bond, which was already done in \cite{fgn3}.

From Mitoma's criterion,  the sequence $\{ \mc Y_t^{n} \;; \;t \in [0,T]\}_{n \in \bb N}$ is tight if we prove tightness for the sequence of real-valued processes $\{ \mc Y_t^{n}(\vphi) \;; \;t \in [0,T]\}_{n \in \bb N}$ for every $\vphi \in \mc S_\beta(\bb R)$. In view of the decomposition \eqref{eq:decomposition}, we are reduced to prove the tightness of the four sequences: \[\begin{array}{ll} \{ \mc Y_0^{n}(\vphi)\}_{n \in \bb N}, & \quad \{ \mc M_t^{n}(\vphi)\; ;\; t \in [0,T]\}_{n \in \bb N}, \\
~\\
\{ \mc I_t^{n}(\vphi)+\mc R_t^{n}(\vphi) \;;\; t \in [0,T]\}_{n \in \bb N},& \quad    \{ \mc B_t^{n}(\vphi) \;;\; t \in [0,T]\}_{n \in \bb N}.\end{array}\]

\subsection{Tightness for $\{ \mc Y_0^{n}(\vphi)\}_{n \in \bb N}$ } This is the simplest: after computing the characteristic function of $\mc Y_0^{n}(\vphi)$, one can easily check that $\mc Y_0^{n}$ converges in distribution to a Gaussian field $\mc Y_0$ with zero mean and covariances given on $\vphi, \psi \in \mc S_\beta(\bb R)$ by
\[
\bb E_\rho\big[\mc Y_0(\vphi) \mc Y_0(\psi)\big]=\chi(\rho) \int_{\bb R} \vphi(u)\psi(u) du.
\]

\subsection{Tightness for $\{ \mc M_t^{n}(\vphi) \;;\; t \in [0,T]\}_{n \in \bb N}$ } From the Aldous' criterion \cite[Proposition 12]{gj2014}, we fix a stopping time $\tau$ bounding by $T$. We have, for any $\lambda >0$,
\begin{align*}
\bb P_\rho \Big[ \big| &\mc M_{\tau+\lambda}^{n}(\vphi)-\mc M_\tau^{n}(\vphi)\big|>\varepsilon\Big] \\ & \le \frac{1}{\varepsilon^2} \bb E_\rho\Big[ \big(\mc M_{\tau+\lambda}^{n}(\vphi)-\mc M_\tau^{n}(\vphi)\big)^2\Big]\\
& \le \frac{1}{\varepsilon^2} \bb E_\rho\bigg[ \int_\tau^{\tau+\lambda} \frac{1}{2n}  \sum_{x\neq -1} \tau_x G_n(\eta_{sn^2}) \big(\nabla_x^n \vphi\big)^2\; ds + \int_\tau^{\tau+\lambda} \frac{1}{2n} H_n(\eta_{sn^2})  \big(\nabla_{-1}^n \vphi\big)^2\; ds \bigg]\\
& \le \frac{\lambda}{\varepsilon^2}\frac{1}{n} \sum_{x \neq 1} \big(\nabla_x^n \vphi\big)^2 + \frac{\lambda \alpha }{\varepsilon^2} \; \frac n {n^{\beta}} \Big\{\vphi\Big(\frac{0}{n}\Big)-\vphi\Big(\frac{-1}{n}\Big)\Big\}^2+ \frac{\lambda a }{\varepsilon^2} \; \frac n {n^\gamma} \Big\{\vphi\Big(\frac{0}{n}\Big)-\vphi\Big(\frac{-1}{n}\Big)\Big\}^2.
\end{align*}
By using the same arguments as in the proof of Lemma \ref{lem:quadratic-var}, and since $\vphi \in \mc S_\beta(\bb R)$, one can prove that the right-hand side above converges to $({\lambda}/{\varepsilon^2})\|\nabla_\beta\vphi\|_{2,\beta}^2$ as $n \to \infty$.
We also have
\[
\bb E_\rho \Big[ \big(\mc M_t^{n}(\vphi)\big)^2 \Big] \le  t\chi(\rho) \|\nabla_\beta\vphi\|_{2,\beta}^2.
\]
Therefore, for any fixed time $t \in [0,T]$, the sequence $\{ \mc M_t^{n}(\vphi)\}_{n \in \bb N}$ is uniformly bounded in ${\bf L}^2(\bb P_\rho)$. The tightness of  $\{ \mc M_t^{n}(\vphi) \;;\; t \in [0,T]\}_{n \in \bb N}$  follows from Aldous' criterion, and we also have that any limit point of the sequence $\{ \mc M_t^{n}(\vphi)\; ;\; t \in [0,T]\}_{n \in \bb N}$ is concentrated on continuous trajectories.

\subsection{Tightness for $\{ \mc I_t^{n}(\vphi)+\mc R_t^n(\vphi)\; ; \;t \in [0,T]\}_{n \in \bb N}$ }
Recall that
\[\mc I_t^{n}(\vphi)+\mc R_t^n(\vphi)
:= \int_0^t \frac{1}{\sqrt n} \sum_{x \in \bb Z} n^2 \bb L_n^S\vphi\Big(\frac{x}{n}\Big) \eta_{sn^2}(x) \; ds.
\]
Then, the proof is exactly the same as in \cite[Section 3.2]{fgn3}, since the above quantity does not depend on $\gamma$.

\subsection{Tightness for $\{ \mc B_t^{n}(\vphi)\; ;\; t \in [0,T]\}_{n \in \bb N}$ } \label{tightB}


Here we only have to prove tightness in the case $\beta\leq 1/2$ and $\gamma=1/2$, since in the other cases, by Proposition \ref{vanishingofB},  it gives no contribution in the limit in $\bf L^2(\bb P_\rho)$. 

Recall the bound \eqref{eq:boundB} that we have obtained for $\bb E_\rho[(\mc B_t^n(\varphi))^2]$. Since $\gamma=1/2$, this bound becomes
\[
\mathbb{E}_{\rho} \Big[ \big(\mc B_t^{n}(\vphi)\big)^2\Big] \le Ct \Big\{\frac{L}{n}+\frac{n^\beta}{\alpha n}\left(1+\big(\log_2(L)\big)^2\right)+\frac{tn}{L^2}+\frac{tn}{L} \Big\}\|\nabla^n\vphi\|_{2,n}^2.
\]
We choose $L$ equal to the integer part of $n \sqrt t$. The quantity
\[
\frac{tn^\beta}{\alpha n}\left(1+\big(\log_2(L)\big)^2\right)
\]
vanishes as $n\to\infty$ because $\beta < 1$, and we obtain that there exists a constant $K$ (that does not depend on $\vphi$) such that, for all $t \ge (1/n^2)$,
\[
\mathbb{E}_{\rho} \Big[ \big(\mc B_t^{n}(\vphi)\big)^2\Big] \le K t^{3/2} \| \nabla^n\vphi\|_{2,n}^2.
\]
For small times $t \le (1/n^2)$, we use a simple Cauchy-Schwarz inequality together with the independence property. We obtain:
\[
\mathbb{E}_{\rho} \Big[ \big(\mc B_t^{n}(\vphi)\big)^2\Big] \le t^2 n  \| \nabla^n\vphi\|_{2,n}^2 \le t^{3/2}  \| \nabla^n\vphi\|_{2,n}^2.
\]
Since the process $\{\eta_{tn^2}\; ;\; t \in [0,T] \}$ is stationary, the following estimate follows:
\[
\bb E_\rho  \Big[ \big(\mc B_t^{n}(\vphi)-\mc B_s^{n}(\vphi)\big)^2\Big] \le K |t-s|^{3/2}   \| \nabla^n\vphi\|_{2,n}.
\]

\subsection{Characterization of limit points}
\label{ssec:limitpoints}
Provided with the respective tightness of the processes, we proceed to the proof of Theorem \ref{thm:crossover} and  Theorem \ref{theo:fluct}. We closely follow \cite{fgn3} (for Theorem \ref{thm:crossover}) and \cite{gj2014} (for Theorem \ref{theo:fluct})  and only give the main arguments.

From the tightness of the four sequences below, we can consider (up to extraction) that $\{\mc Y_t^{n}\;;\;t\in[0,T]\}$, $\{\mc M_t^{n}\;;\;t\in[0,T]\}$, $\{\mc I_t^n + \mc R_t^n\;;\;t\in[0,T]\}$, and $\{\mc B_t^{n}\;;\;t\in[0,T]\}$ converge as $n\to\infty$ to  $\{\mc Y_t\;;\;t\in[0,T]\}$, $\{\mc M_t\;;\;t\in[0,T]\}$, $\{\mc I_t + \mc R_t\;;\;t\in[0,T]\}$, and $\{\mc B_t\;;\;t\in[0,T]\}$, respectively.

\begin{proof}[Proof of Theorem \ref{thm:crossover}]

From Proposition \ref{vanishingofB} and Lemma \ref{lem:quadratic-var}, it is not difficult to show (see \cite{fgn3} for details) that $\{ \mc Y_t \; ; \; t \in [0,T]\}$ is in $\mc C([0,T],\mc S'_\beta(\bb R))$, and also that, for $\vphi \in \mc S_\beta(\bb R)$,
\[\mc M_t(\vphi)=\mc Y_t(\vphi)-\mc Y_0(\vphi) - \frac12\int_0^t \mc Y_s(\Delta_\beta(\vphi))\;ds\]
is a martingale of quadratic variation given by $t\chi(\rho)\|\nabla_\beta \vphi\|_{2,\beta}^2$ (using Lemma \ref{lem:quadratic-var}).  Then, Theorem \ref{thm:crossover} is a direct consequence of Proposition \ref{prop:sol_ou}.
\end{proof}

\begin{proof}[Proof of Theorem \ref{theo:fluct}]

We are now focusing on the case $\beta \leq 1/2$ and $\gamma=1/2$, hence we work with the usual Schwartz space $\mc S(\bb R)$ and the usual operators $\nabla$ and $\Delta$. We refer the reader to \cite{gj2014} to see that:
\begin{itemize}
\item the process $\{\mc Y_t\;;\;t\in[0,T]\}$ has continuous trajectories with respect to the strong topology of $\mc S'(\bb R)$,
\item the process $\{\mc Y_t\;;\;t\in[0,T]\}$ is stationary.
\end{itemize}
The key ingredients to prove Theorem \ref{theo:fluct} are \eqref{eq:ec1} and \eqref{eq:ec2}. An easy consequence of Theorem  \ref{theo:BG} (which can be derived as in Subsection \ref{vanishB} after taking $\ell=L=\varepsilon n$) is the following: there exists a constant $K$ such that, for any $\vphi \in \mc S(\bb R)$,
\[\bb E_\rho\bigg[\Big(\mc B_t(\vphi)-\mc B_s(\vphi) - \chi(\rho)\mc A_{s,t}^\varepsilon(\vphi)\Big)^2\bigg]\le K(t-s)\varepsilon  \|\nabla\vphi\|_2^2.\]
 Keeping this inequality in mind, then adding and subtracting $\chi(\rho)^{-1}(\mc B_t(\vphi)-\mc B_s(\vphi))$ inside the square in
$
\bb E_\rho[ (\mc A_{s,t}^\varepsilon(\vphi)-\mc A_{s,t}^\delta(\vphi))^2]
$   we are lead to the the desired inequality in \eqref{eq:ec2}.

 Moreover, to prove \eqref{eq:ec1} it is enough to have
\begin{equation}\label{eq:energyestimate}
\lim_{n\to\infty} \bb E_\rho \Big[\big(\mc I_t^{n}(\vphi)\big)^2\Big]=\lim_{n\to\infty}\bb E_{\rho}\bigg[\Big(\frac12\int_0^t \mc Y_s^n(\Delta \vphi) \;ds\Big)^2\bigg] \le \kappa\, t  \|\nabla\vphi\|_2^2.
\end{equation}
Let us decompose the line $\bb Z$ into boxes of size $\ell\in\bb N$. We have
\begin{align}
\int_0^t \mc Y_s^n(\Delta \vphi) \;ds & = \int_0^t \frac{1}{\sqrt n} \sum_{x \in \bb Z} \Delta\vphi\Big(\frac{x}{n}\Big) \bar\eta_{sn^2}(x) \; ds\notag\\
&=\int_0^t \frac{1}{\sqrt n} \sum_{j \in \bb Z}\sum_{k=1}^{\ell} \Big[\Delta\vphi\Big(\frac{j\ell+k}{n}\Big) - \Delta\vphi\Big(\frac{j\ell}{n}\Big)\Big]  \bar\eta_{sn^2}(j\ell+k) \; ds\label{eq:decomp1}\\
& \quad + \int_0^t \frac{1}{\sqrt n} \sum_{j \in \bb Z}\Big[\Delta\vphi\Big(\frac{j\ell}{n}\Big) \sum_{k=1}^{\ell} \bar\eta_{sn^2}(j\ell+k) \Big]\; ds.\label{eq:decomp2}
\end{align}
Due to the smoothness and the fast decaying of $\vphi$, by the Cauchy-Schwarz inequality, we can deal with \eqref{eq:decomp1} as follows:
\begin{multline*}\bb E_{\rho}\bigg[\Big(\int_0^t \frac{1}{\sqrt n} \sum_{j \in \bb Z}\sum_{k=1}^{\ell} \Big[\Delta\vphi\Big(\frac{j\ell+k}{n}\Big) - \Delta\vphi\Big(\frac{j\ell}{n}\Big)\Big]  \bar\eta_{sn^2}(j\ell+k) \; ds\Big)^2\bigg] \\ \leq Ct^2 \frac{1}{n} \sum_{j \in \bb Z}\sum_{k=1}^{\ell} \Big[\Delta\vphi\Big(\frac{j\ell+k}{n}\Big) - \Delta\vphi\Big(\frac{j\ell}{n}\Big)\Big]^2 \leq C't^2 \frac{\ell^2}{n^2},   \end{multline*}
where $C,C'>0$ are real positive constants that do not depend on $\ell$ nor $n$. We now turn to \eqref{eq:decomp2}. Recall from \eqref{eq:mean} the definition of  $\vec{\eta}^{\ell}(x)$. It holds that:
\begin{align*}
 \bb E_{\rho}\bigg[\Big(\int_0^t & \frac{1}{\sqrt n} \sum_{j \in \bb Z}\Big[\Delta\vphi\Big(\frac{j\ell}{n}\Big)  \sum_{k=1}^{\ell} \bar\eta_{sn^2}(j\ell+k) \Big]\; ds \Big)^2\bigg]  \\ & =\bb E_{\rho}\bigg[\Big( \int_0^t \frac{\ell}{\sqrt n} \sum_{j \in \bb Z}\Delta\vphi\Big(\frac{j\ell}{n}\Big)  \vec\eta^\ell_{sn^2}(j\ell) \; ds\Big)^2\bigg]  \\
 &  =  \bb E_{\rho}\bigg[\Big( \int_0^t \frac{\ell}{\sqrt n} \sum_{j \in \bb Z} \frac{n}{\ell} \Big[\nabla\vphi\Big(\frac{j\ell}{n}\Big)-\nabla\vphi\Big(\frac{j\ell-\ell}{n}\Big)\Big]  \vec\eta^\ell_{sn^2}(j\ell) \; ds \Big)^2\bigg] + O\Big(\frac{\ell^2}{n^2}\Big).
\end{align*}
After summing by parts, and taking $\ell=\varepsilon n$ ($\varepsilon >0)$, we deduce that
\begin{multline*}
\bb E_{\rho}\bigg[\Big(\int_0^t \mc Y_s^n(\Delta \vphi)\; ds\Big)^2\bigg] \\=\bb E_{\rho}\bigg[\Big( \int_0^t \sqrt n \sum_{j \in \bb Z}  \nabla\vphi\Big(\frac{j\varepsilon n}{n}\Big)  \Big\{\vec\eta^{\varepsilon n}_{sn^2}(j\varepsilon n) - \vec\eta^{\varepsilon n}_{sn^2}((j+1)\varepsilon n)\Big\}  \; ds \Big)^2\bigg] + O(\varepsilon^2).
\end{multline*}
Therefore, for proving \eqref{eq:energyestimate} it is enough to show that
\[
\limsup_{\varepsilon \to 0} \limsup_{n \to \infty} \mathbb{E}_{\rho}\bigg[\Big(\sqrt n\int_{0}^t  \sum_{x\in\varepsilon n \mathbb{Z}}\nabla\vphi\Big(\frac{x}{n}\Big)\Big\{\vec{\eta}_{sn^2}^{\varepsilon n}(x)-\vec{\eta}_{sn^2}^{\varepsilon n}(x+\varepsilon n)\Big\}\;ds\Big)^2\bigg] \leq \kappa t \|\nabla\vphi\|_2^2.
\]
Last inequality above is a direct consequence of Corollary \ref{cor:estimate}, stated and proved in Section \ref{sec:auxi}.
\end{proof}

\section{Proof of the Boltzmann-Gibbs Principle}\label{sec:BG}

In this section we present a proof of  the second-order Boltzmann-Gibbs principle, stated in Theorem \ref{theo:BG}, which is  the main technical difficulty in this work. For that purpose we recall \eqref{eq:mean} and similarly, for $L\in\mathbb{N}$, we define $\vecleft\eta^L(x)$ as the empirical average to the left of the site $x$, that is:
\[
\vecleft\eta^L(x):=\frac{1}{L}\sum_{y=x-L}^{x-1} \bar\eta(y).
\]
 The main idea to prove   Theorem \ref{theo:BG} consists in introducing averages over boxes of a certain intermediate size $\ell$, until reaching the desired box of size $L$. To achieve that goal a multi-scale analysis is done for a particular function, starting at the initial size $\ell_0$ which does not depend on $n$. In order to prove Theorem \ref{theo:BG}, we use the following decomposition:
\begin{align}&\bar{\eta}(x)\bar{\eta}(x+1) -\big(\vec{\eta}^L(x)\big)^2+\frac{\chi(\rho)}{L}\notag\\
&=\bar{\eta}(x)\big(\bar{\eta}(x+1)-\vec{\eta}^{\ell_0}(x)\big) \label{eq:term1} \vphantom{\Big(}\\
& \quad + \vec\eta^{\ell_0}(x) \big(\eta(x)-\vecleft\eta^{\ell_0}(x)\big) \label{eq:term1bis} \vphantom{\Big(}\\
& \quad + \vecleft\eta^{\ell_0}(x) \big( \vec\eta^{\ell_0}(x)-\vec\eta^L(x)\big) \label{eq:term3}\vphantom{\Big(}\\
& \quad +\vec\eta^L(x) \big(\vecleft\eta^{\ell_0}(x)-\eta(x)\big) \label{eq:term4} \vphantom{\Big(}\\
&\quad +\vec{\eta}^L(x)\bar{\eta}(x)
-\big(\vec{\eta}^L(x)\big)^2+ \frac{\big(\bar\eta(x)-\bar\eta(x+1)\big)^2}{2L}\label{eq:term5}\vphantom{\Big(}\\
&\quad -\frac{\big(\bar\eta(x)-\bar\eta(x+1)\big)^2}{2L}+\frac{\chi(\rho)}{L}.\label{eq:term6}\vphantom{\Big(}
\end{align}
The decomposition above involves six main terms, which we treat separately. Let us sketch an outline of the section: \begin{itemize}
\item the term \eqref{eq:term1} is estimated in Subsection \ref{ssec:oneblock} by what we wall \emph{the one-block estimate} (Proposition \ref{prop:one-block}). With a very similar argument, both terms \eqref{eq:term1bis} and \eqref{eq:term4} can also be worked out, see Proposition \ref{prop:one-block-left};  \item the term \eqref{eq:term3} is the most trickiest one, for which we need to perform a \emph{multi-scale analysis}, presented in Subsection \ref{ssec:multi};
\item a simple Cauchy-Schwarz estimate allows to control  \eqref{eq:term6} and is exposed in Subsection \ref{ssec:cs};
\item finally, \eqref{eq:term5}  is treated separately in Subsection \ref{ssec:decomp}.
\end{itemize}

The main idea that permits to obtain sharp bounds consists in counting very carefully the number of times we need to cross the slow bond $\{-1,0\}$ when we make  replacements of type $\vec\eta^\ell(x) \mapsto \vec\eta^L(x)$ for some $\ell,L\in\bb N$ (for instance). For that purpose, and to facilitate the reading, we introduce the notation 
\[
\Lambda_{y}^{\ell}:=\{-\ell-y,\ldots,-y-1\}, \qquad \ell \in \bb N, y \in \bb Z.
\]
To keep the notation simple in the following argument, we let $C=C(\rho)$  denote  a constant (that does  not depend on $n$ nor on $t$ nor on the sizes of the boxes involved) that may change from line to line.  
In all what follows, $v:\mathbb{Z}\to{\mathbb{R}}$ is a function satisfying \eqref{vsummable}. Along the proofs we will consider several finite boxes in $\bb Z$.

\subsection{Estimate of \eqref{eq:term1} \eqref{eq:term1bis} and \eqref{eq:term4}: One-block estimates} \label{ssec:oneblock}

To bound  \eqref{eq:term1} we use the following:

\begin{proposition} [One-block estimate to the right]
\label{prop:one-block}
Let $\ell_0\in\bb {N}$ and $\psi:\Omega\to\mathbb{R}$ a local function whose support does not intersect the set of points $\{1,\dots,\ell_0\}$.  We assume that $\psi$ has mean zero with respect to $\nu_\rho$ and we denote by $\mathrm{Var}(\psi)$ its variance.

 Then, for any $t>0$:
\begin{multline*}
\mathbb{E}_{\rho}\bigg[\Big(\int_{0}^t  \sum_{x\in\mathbb{Z}}v(x)\tau_x\psi(\eta_{sn^2})\big(\bar{\eta}_{sn^2}(x+1)-\vec{\eta}_{sn^2}^{\ell_0}(x)\big)\;ds\Big)^2\bigg] \\
\leq
{C(\rho)t}{\mathrm{Var}(\psi)}\bigg(\frac{\ell_0^2}{n}\|v\|_{2,n}^2+\frac{\ell_0 n^\beta}{n^2\alpha}\sum_{x \in \Lambda_1^{\ell_0-1}} v^2(x)\bigg),
\end{multline*}
recalling that $\Lambda_1^{\ell_0-1}=\{-\ell_0,\ldots,-2\}$.
\end{proposition}

\begin{remark} In particular, notice that
\[\sum_{x \in \Lambda_1^{\ell_0-1}} v^2(x) \le \sum_{x\neq -1} v^2(x).\]
\end{remark}

\begin{proof}
By  \cite[Lemma 2.4]{KLO}, we can bound the previous expectation from above by
\begin{equation*}
 Ct\bigg\|\sum_{x\in\mathbb{Z}}v(x)\tau_x\psi(\eta)\big(\bar{\eta}(x+1)-\vec{\eta}^{\ell_0}(x)\big)\bigg\|_{-1}^2,\end{equation*}
where the $H_{-1}$-norm is defined through a variational formula, and in particular the previous expression is equal to
\begin{equation*}
Ct\sup_{f\in {\bf L}^2(\nu_\rho)}\bigg\{2\int \sum_{x\in\mathbb{Z}}v(x)\tau_x\psi(\eta)\big(\bar{\eta}(x+1)-\vec{\eta}^{\ell_0}(x)\big)f(\eta)\nu_\rho(d\eta)-n^2D_n(f)\bigg\},
\end{equation*}
where $D_n(f)$ is the Dirichlet form associated to the Markov process, and is defined as $D_n(f)=\int f(\eta)\mathcal{L}_nf(\eta)\nu_\rho(d\eta)$, see \eqref{eq:dirichlet}.
Now we notice that
\begin{equation*}
\bar{\eta}(x+1)-\vec{\eta}^{\ell_0}(x)=\frac{1}{\ell_0}\sum_{y=x+2}^{x+\ell_0}\sum_{z=x+1}^{y-1}(\bar{\eta}(z)-\bar{\eta}(z+1)).
\end{equation*}
 Now, we write the integral above as twice its half and in one of the terms we make the exchange $\eta$ to $\eta^{z,z+1}$, for which the measure $\nu_\rho$ is invariant. Since the support of $\tau_x\psi$ does not intersect this set of points, it also remains invariant, and we get
\begin{equation*}
\begin{split}
&2\int\sum_{x\in\mathbb{Z}}v(x)\tau_x\psi(\eta)\Big\{\frac{1}{\ell_0}\sum_{y=x+2}^{x+\ell_0}\sum_{z=x+1}^{y-1}(\bar{\eta}(z)-\bar{\eta}(z+1))\Big\}f(\eta)\nu_\rho(d\eta)\\
=\; &2\int\sum_{x\in\mathbb{Z}}v(x)\tau_x\psi(\eta)\Big\{\frac{1}{\ell_0}\sum_{y=x+2}^{x+\ell_0}\sum_{z=x+1}^{y-1}(\bar{\eta}(z+1)-\bar{\eta}(z))\Big\}f(\eta^{z,z+1})\nu_\rho(d\eta)\\
=\; &\int\sum_{x\in\mathbb{Z}}v(x)\tau_x\psi(\eta)\Big\{\frac{1}{\ell_0}\sum_{y=x+2}^{x+\ell_0}\sum_{z=x+1}^{y-1}(\bar{\eta}(z)-\bar{\eta}(z+1))\Big\}(f(\eta)-f(\eta^{z,z+1}))\nu_\rho(d\eta).
\end{split}
\end{equation*}
At this point we have to be careful and to split the sum over $x$, according to the points for which the slow bond $\{-1,0\}$ belongs to the following set of bonds \begin{equation}\label{eq:set}\mathcal{E}_x^{\ell_0}:=\big\{(z,z+1):\,x+1\leq z\leq y-1\, \text{ and }\, x+2\leq y\leq x+\ell_0\big\}.\end{equation} For that purpose, recall that $\Lambda_1^{\ell_0-1}=\{-\ell_0,\ldots,-2\}$.
A simple computation shows that for $x\in\Lambda_1^{\ell_0-1}$ the slow bond belongs to the set $\mathcal{E}_x^{\ell_0}$, otherwise, it does not.
According to this observation, the last integral can be written as the sum of
\begin{equation}
\label{eq:auxest1}
\int\sum_{x\notin\Lambda_1^{\ell_0-1}}v(x)\tau_x\psi(\eta)\Big\{\frac{1}{\ell_0}\sum_{y=x+2}^{x+\ell_0}\sum_{z=x+1}^{y-1}(\bar{\eta}(z)-\bar{\eta}(z+1))\Big\}(f(\eta)-f(\eta^{z,z+1}))\nu_\rho(d\eta) 
\end{equation} 
and 
\begin{equation}
\label{eq:auxest11}
\int\sum_{x\in\Lambda_1^{\ell_0-1}}v(x)\tau_x\psi(\eta)\Big\{\frac{1}{\ell_0}\sum_{y=x+2}^{x+\ell_0}\sum_{z=x+1}^{y-1}(\bar{\eta}(z)-\bar{\eta}(z+1))\Big\}(f(\eta)-f(\eta^{z,z+1}))\nu_\rho(d\eta).
\end{equation}
By Young's inequality, for any $(A_x)_{x\in\mathbb{Z}}$ of positive real numbers,  \eqref{eq:auxest1} is bounded by
\begin{equation*}
\begin{split}
&\frac{1}{\ell_0}\sum_{x\notin\Lambda_1^{\ell_0-1}}\sum_{y=x+2}^{x+\ell_0}\sum_{z=x+1}^{y-1}v(x)\frac{A_x}{2}\int(\tau_x\psi(\eta))^2(\bar{\eta}(z)-\bar{\eta}(z+1))^2\nu_\rho(d\eta)\\
+&\frac{1}{\ell_0}\sum_{x\notin\Lambda_1^{\ell_0-1}}\sum_{y=x+2}^{x+\ell_0}\sum_{z=x+1}^{y-1}\frac{v(x)}{2A_x}I_{z,z+1}(f),
\end{split}
\end{equation*}
where we define \begin{equation}I_{z,z+1}(f):=\int\big(f(\eta)-f(\eta^{z,z+1})\big)^2\nu_\rho(d\eta).\label{eq:def_i}\end{equation}
By choosing, for each $x\in\bb Z$, $2A_x=\ell_0 v(x)/n^2$ and by independence, the first term above is bounded by
\begin{equation}\label{error1}
C(\rho)\frac{\mathrm{Var}(\psi)}{n^2}\sum_{x\notin\Lambda_1^{\ell_0-1}}\sum_{y=x+2}^{x+\ell_0}\sum_{z=x+1}^{y-1}v^2(x)\leq C(\rho)\mathrm{Var}(\psi)\frac{\ell_0^2}{n^2}\sum_{x\notin\Lambda_1^{\ell_0-1}}v^2(x),
\end{equation}
for some positive constant $C(\rho)$. The second one is bounded from above by
\begin{equation}\label{error2}
\frac{n^2}{\ell_0^2}\sum_{x\notin\Lambda_1^{\ell_0-1}}\sum_{y=x+2}^{x+\ell_0}\sum_{z=x+1}^{y-1}I_{z,z+1}(f).
\end{equation}
Now, we use again Young's inequality to bound the second integral  \eqref{eq:auxest11} by
\begin{align}
&\frac{1}{\ell_0}\sum_{x\in\Lambda_1^{\ell_0-1}}\sum_{y=x+2}^{x+\ell_0}\sum_{z=x+1}^{y-1}\frac{v(x)A_x}{2\Xi^n_{z,z+1}}\int(\tau_x\psi(\eta))^2(\bar{\eta}(z)-\bar{\eta}(z+1))^2\nu_\rho(d\eta) \label{eq:auxest2}\\
+&\frac{1}{\ell_0}\sum_{x\in\Lambda_1^{\ell_0-1}}\sum_{y=x+2}^{x+\ell_0}\sum_{z=x+1}^{y-1}\frac{v(x)\Xi^n_{z,z+1}}{2A_x}I_{z,z+1}(f). \label{eq:auxest22}
\end{align}
Notice that we have introduced the positive numbers $\Xi^n_{z,z+1}$ which are defined below in \eqref{eq:zeta} and correspond to the coefficients of the Dirichlet form \eqref{eq:dirichlet}. These numbers have to be added in the case $x\in\Lambda_1^{\ell_0-1}$ for which case the slow bond belongs to the set $\mathcal{E}_x^{\ell_0}$.
By taking, for each $x\in\bb Z$, $2A_x=\ell_0 v(x)/n^2$, the first term \eqref{eq:auxest2} can be bounded by
\begin{equation}\label{expextra}
\frac{\mathrm{Var}(\psi)}{n^2}\sum_{x\in\Lambda_1^{\ell_0-1}}\sum_{y=x+2}^{x+\ell_0}\sum_{z=x+1}^{y-1}\frac{v^2(x)}{\Xi^n_{z,z+1}}C(\rho).
\end{equation}
Now, we remark that if $x=-\ell_0$ then the slow bond appears only once in $\mathcal{E}_x^{\ell_0}$, but for $x=-\ell_0+1$ then the slow bond appears twice in $\mathcal{E}_x^{\ell_0}$, and so on, and finally for $x=-2$, the slow bond appears $\ell_0-1$ times in $\mathcal{E}_x^{\ell_0}$.
Therefore, we can bound the previous sum by
\begin{equation}\label{error3}
\frac{C(\rho)}{\Xi^n_{-1,0}}\big[v^2(-\ell_0)+2v^2(-\ell_0+1)+\cdots+(\ell_0-1) v^2(-2)\big]+\frac{C(\rho)}{n^2}\sum_{x\in\Lambda_1^{\ell_0-1}}\sum_{y=x+2}^{x+\ell_0}\sum_{\substack{z=x+1\\z\neq -1}}^{y-1}\frac{v^2(x)}{\Xi^n_{z,z+1}}
\end{equation}
so that \eqref{expextra} is bounded from above by 
\[
C(\rho)\mathrm{Var}(\psi)\Big(\frac{\ell_0 n^\beta}{n^2\alpha}+\frac{\ell_0^2}{n^2}\Big)\sum_{x\in\Lambda_1^{\ell_0-1}}v^2(x).
\]
The second term  \eqref{eq:auxest22}
is bounded by
\begin{equation}\label{error4}
\frac{n^2}{\ell_0^2}\sum_{x\in\Lambda_1^{\ell_0-1}}\sum_{y=x+2}^{x+\ell_0}\sum_{z=x+1}^{y-1}\Xi^n_{z,z+1}I_{z,z+1}(f).
\end{equation}
Putting together \eqref{error1}, \eqref{error2}, \eqref{error3} and \eqref{error4}, the integral in the statement of Proposition \ref{prop:one-block} is bounded from above by
\begin{equation*}
C(\rho)\mathrm{Var}(\psi)\bigg(\frac{\ell_0^2}{n}\|v\|_{2,n}^2+\frac{\ell_0 n^\beta}{n^2\alpha}\sum_{x \in \Lambda_1^{\ell_0-1}} v^2(x)\bigg)
+\frac{n^2}{\ell_0^2}\sum_{x\in\mathbb{Z}}\sum_{y=x+2}^{x+\ell_0}\sum_{z=x+1}^{y-1}\Xi^n_{z,z+1}I_{z,z+1}(f).
\end{equation*}
We now refer to Lemma \ref{lemma:dir1} below, which gives estimates on the Dirichlet form: by \eqref{eq:dirich1} applied with $\ell=\ell_0-1$ the result follows. 

\end{proof}

This one-block estimate is enough to control \eqref{eq:term1} by  taking $\tau_x\psi(\eta)=\bar{\eta}(x)$. To treat the remaining terms  \eqref{eq:term1bis} and \eqref{eq:term4}, since the averages are taken to the left of the site $x$, we need to adapt the previous argument, as given in Proposition \ref{prop:one-block-left} below.

\begin{proposition}[One-block estimate to the left]\label{prop:one-block-left}
Let $\ell_0\in\bb {N}$ and $\psi:\Omega\to\mathbb{R}$ a local function whose support does not intersect the set of points $\{-\ell_0,\dots,0\}$. We assume that $\psi$ has mean zero with respect to $\nu_\rho$ and we denote by $\mathrm{Var}(\psi)$ its variance.

 Then, for any $t>0$:
\begin{multline*}
\mathbb{E}_{\rho}\bigg[\Big(\int_{0}^t  \sum_{x\in\mathbb{Z}}v(x)\tau_x\psi(\eta_{sn^2})\big(\bar{\eta}_{sn^2}(x)-\vecleft{\eta}_{sn^2}^{\ell_0}(x)\big)\;ds\Big)^2\bigg] \\
\leq
{C(\rho)t}{\mathrm{Var}(\psi)}\bigg(\frac{\ell_0^2}{n}\|v\|_{2,n}^2+\frac{\ell_0 n^\beta}{n^2\alpha}\sum_{x \in \Lambda_{-\ell_0}^{\ell_0}} v^2(x)\bigg),
\end{multline*}
recalling that $\Lambda_{-\ell_0}^{\ell_0}=\{0,\ldots,\ell_0-1\}$.
\end{proposition}

\begin{remark}
It still holds that
\[\sum_{x \in \Lambda_{-\ell_0}^{\ell_0}} v^2(x) \le \sum_{x\neq -1} v^2(x).\]
\end{remark}

\begin{proof} Since the proof is very similar to the previous one, we only give the main arguments: 
we have to control
\begin{equation*}
 Ct\bigg\|\sum_{x\in\mathbb{Z}}v(x)\tau_x\psi(\eta)\big(\bar{\eta}(x)-\vecleft{\eta}^{\ell_0}(x)\big)\bigg\|_{-1}^2,\end{equation*}
and we notice that
\begin{equation*}
\bar{\eta}(x)-\vecleft{\eta}^{\ell_0}(x)=\frac{1}{\ell_0}\sum_{y=x-\ell_0}^{x-1}\sum_{z=y}^{x-1}(\bar{\eta}(z+1)-\bar{\eta}(z)).
\end{equation*}
 As before, when we split the make the exchange $\eta$ to $\eta^{z,z+1}$ (for which  $\nu_\rho$ and $\tau_x\psi$  remain invariant) in the integrals which are involved,  we have to be careful and to split the sum over $x$, according to the points for which the slow bond $\{-1,0\}$ belongs to the following set of bonds \begin{equation*}\widetilde{\mathcal{E}}_x^{\ell_0}:=\big\{(z,z+1):\,y\leq z\leq x-1\, \text{ and }\, x-\ell_0\leq y\leq x-1\big\}.\end{equation*} Here, recall that $\Lambda_{-\ell_0}^{\ell_0}=\{0,\ldots,\ell_0-1\}$.
A simple computation shows that for $x\in\Lambda_{-\ell_0}^{\ell_0}$ the slow bond belongs to the set $\widetilde{\mathcal{E}}_x^{\ell_0}$, otherwise, it does not.
Afterwards, the argument is straightforwardly identical to the proof of the previous proposition. One can conclude that the integral in the statement of Proposition \ref{prop:one-block-left} is bounded from above by
\begin{equation*}
C(\rho)\mathrm{Var}(\psi)\bigg(\frac{\ell_0^2}{n}\|v\|_{2,n}^2+\frac{\ell_0 n^\beta}{n^2\alpha}\sum_{x \in \Lambda_{-\ell_0}^{\ell_0}} v^2(x)\bigg)
+\frac{n^2}{\ell_0^2}\sum_{x\in\mathbb{Z}}\sum_{y=x-\ell_0}^{x-1}\sum_{z=y}^{x-1}\Xi^n_{z,z+1}I_{z,z+1}(f).
\end{equation*}
Now, by Lemma \ref{lemma:dir1}, precisely \eqref{eq:dirich2}, the result follows. 
\end{proof}

By Proposition \ref{prop:one-block-left}, the terms \eqref{eq:term1bis} and \eqref{eq:term4} are controlled using $\tau_x\psi(\eta)=\vec\eta^{\ell}(x+1)$ for some $\ell \in \bb N$.
Finally, the sum of \eqref{eq:term1} \eqref{eq:term1bis} and \eqref{eq:term4} gives a total error contribution bounded by
\begin{equation*}
Ct\Big( 1+\frac{1}{\ell_0}+\frac{1}{L}  \Big) \bigg(\frac{\ell_0^2}{n} \big\|v\big\|_{2,n}^2 + \frac{\ell_0 n^\beta}{n^2 \alpha} \sum_{x\neq -1} v^2(x) \bigg).
\end{equation*}
Since $\ell_0$ is supposed to be independent of $n$, this bound can be simplified as
\begin{equation*}
Ct \bigg(\frac{1}{n} \big\|v\big\|_{2,n}^2 + \frac{n^\beta}{n^2 \alpha} \sum_{x\neq -1} v^2(x) \bigg).
\end{equation*}

\subsection{Estimate of \eqref{eq:term3}: Multi-scale analysis}
\label{ssec:multi}
The idea behind the estimate of \eqref{eq:term3} is as follows: instead of replacing $\vec\eta^{\ell_0}(x)$ by $\vec\eta^L(x)$ in one step, we do it gradually, by doubling the size of the box of size $\ell_0$ at each step. For that purpose, let $\ell_{k+1}=2\ell_k$ and assume first that $L=2^M\ell_0$ for some $M\in\bb N$.
Then, rewrite \eqref{eq:term3} as 
\begin{align}
\vecleft\eta^{\ell_0}(x) \big(\vec\eta^{\ell_0}(x)-\vec\eta^L(x)\big) 
& = \sum_{k=0}^{M-1} \vecleft\eta^{\ell_k}(x) \big(\vec\eta^{\ell_k}(x)-\vec\eta^{\ell_{k+1}}(x)\big) \label{eq:dec1}\\
&  + \sum_{k=0}^{M-2} \vec\eta^{\ell_{k+1}}(x) \big(\vecleft\eta^{\ell_k}(x)-\vecleft\eta^{\ell_{k+1}}(x)\big) \label{eq:dec2}\\
& + \vec\eta^L(x) \big(\vecleft\eta^{\ell_{M-1}}(x)-\vecleft\eta^{\ell_0}(x)\big) \vphantom{\sum_{k}^M}. \label{eq:dec3}
\end{align}
We start with the estimate of the terms that appear in sums \eqref{eq:dec1} and \eqref{eq:dec2}.

\begin{proposition}[Doubling the box]\label{doub box}
Let $\ell_k\in\mathbb{N}$, $\ell_{k+1}=2\ell_k$ and $\psi:\Omega\to\mathbb{R}$ a local function whose support does not intersect the set of points $\{1,\ldots,\ell_{k+1}\}$.  In the same way, let  $\widetilde\psi:\Omega\to\bb R$ be a local function whose support does not intersect the set of points $\{-\ell_{k+1},\ldots,-1\}$.We assume that $\psi$ (resp. $\widetilde\psi$) have mean zero with respect to $\nu_\rho$ and we denote by $\mathrm{Var}(\psi)$ (resp. $\mathrm{Var}(\widetilde\psi)$) its variance.

Then, for any $t>0$:
\begin{multline}
\mathbb{E}_{\rho}\bigg[\Big(\int_{0}^t  \sum_{x\in\mathbb{Z}}v(x)\tau_x\psi(\eta_{sn^2})\big(\vec{\eta}_{sn^2}^{\ell_k}(x)-\vec{\eta}_{sn^2}^{\ell_{k+1}}(x)\big)\;ds\Big)^2\Big]
\\ \leq
 C(\rho)t\mathrm{Var}(\psi)\bigg( \frac{\ell_k^2}{n}\|v\|_{2,n}^2+\frac{n^\beta\ell_k}{n^2\alpha}\sum_{x\neq -1}v^2(x)\bigg).\label{eq:doub1}
\end{multline}
\begin{multline}
\mathbb{E}_{\rho}\bigg[\Big(\int_{0}^t ds \sum_{x\in\mathbb{Z}}v(x)\tau_x\widetilde\psi(\eta_{sn^2})\big(\vecleft{\eta}_{sn^2}^{\ell_k}(x)-\vecleft{\eta}_{sn^2}^{\ell_{k+1}}(x)\big)\;ds\Big)^2\bigg]
\\\leq
 C(\rho)t\mathrm{Var}(\widetilde\psi)\bigg( \frac{\ell_k^2}{n}\|v\|_{2,n}^2+\frac{n^\beta\ell_k}{n^2\alpha}\sum_{x\neq -1}v^2(x)\bigg).\label{eq:doub2}
\end{multline}
\end{proposition}

\begin{proof}
We only prove \eqref{eq:doub1}. The same argument can easily be written for \eqref{eq:doub2}, in the same spirit as for Proposition \ref{prop:one-block-left}.  First, we notice that
\begin{equation*}
\vec{\eta}^{\ell_k}(x)-\vec{\eta}^{\ell_{k+1}}(x)=\frac{1}{2\ell_{k}}\sum_{y=x+1}^{x+\ell_k}(\bar{\eta}(y)-\bar{\eta}(y+\ell_{k})).
\end{equation*}
By \cite[Lemma 2.4]{KLO}, a change of variables $y \mapsto y-x$ and the convexity inequality 
\[(a_1+\cdots+a_\ell)^2 \le \ell (a_1^2+\cdots+ a_\ell^2), \] the expectation in the left-hand side of \eqref{eq:doub1} is bounded from above by
\begin{equation}\label{eq:init}
Ct\ell_k\sum_{y=1}^{\ell_k}\bigg\|\sum_{x\in\mathbb{Z}}v(x)\tau_x\psi(\eta)\frac{1}{2\ell_{k}}(\bar{\eta}(y+x)-\bar{\eta}(y+x+\ell_{k}))\bigg\|_{-1}^2.\end{equation}
By the variational formula for the $H_{-1}$ norm, the quantity inside the sum is equal to
\[
\sup_{f\in {\bf L}^2(\nu_\rho)}\bigg\{2\int \sum_{x\in\mathbb{Z}}v(x)\tau_x\psi(\eta)\frac{1}{2\ell_{k}}(\bar{\eta}(y+x)-\bar{\eta}(y+x+\ell_{k}))f(\eta)\nu_\rho(d\eta)-n^2D_n(f)\bigg\}.
\]
As above, we write
\begin{equation*}
\bar{\eta}(y+x)-\bar{\eta}(y+x+\ell_k)=\sum_{z=y+x}^{y+x+\ell_k-1}(\bar{\eta}(z)-\bar{\eta}(z+1)),
\end{equation*}
and we write the integral in the variational formula above as twice its half  and in one of the terms we make the exchange $\eta$ to $\eta^{z,z+1}$, for which the measure $\nu_\rho$ is invariant. By the imposed conditions on the support of $\psi$ we get that
\begin{equation*}
\begin{split}
&2\int \sum_{x\in\mathbb{Z}}v(x)\tau_x\psi(\eta)\frac{1}{2\ell_k}(\bar{\eta}(y+x)-\bar{\eta}(y+x+\ell_k))f(\eta)\nu_\rho(d\eta)\\
=&\int \sum_{x\in\mathbb{Z}}v(x)\tau_x\psi(\eta)\frac{1}{2\ell_k}\sum_{z=y+x}^{y+x+\ell_k-1}(\bar{\eta}(z)-\bar{\eta}(z+1))(f(\eta)-f(\eta^{z,z+1}))\nu_\rho(d\eta).
\end{split}
\end{equation*}
At this point we have to split the sum in $x$ above, according to the points for which the slow bond $(-1,0)$ belongs to the set of bonds \[\tilde{\mathcal{E}}_{x,y}^{\ell_k}:=\big\{(z,z+1):\,y+x\leq z\leq y+x+\ell_k-1\big\}.\]
A simple computation shows that for $x\in\Lambda_{y}^{\ell_k}=\{-\ell_k-y,\ldots,-y-1\}$ the slow bond belongs to the set $\tilde{\mathcal{E}}_{x,y}^{\ell_k}$, otherwise, it does not.
From this, we can rewrite last integral as
\begin{align}
&\int \sum_{x\notin\Lambda_{y}^{\ell_k}}v(x)\tau_x\psi(\eta)\frac{1}{2\ell_k}\sum_{z=y+x}^{y+x+\ell_k-1}(\bar{\eta}(z)-\bar{\eta}(z+1))(f(\eta)-f(\eta^{z,z+1}))\nu_\rho(d\eta)\label{eq:auxest3}\\
+&\int \sum_{x\in\Lambda_{y}^{\ell_k}}v(x)\tau_x\psi(\eta)\frac{1}{2\ell_k}\sum_{z=y+x}^{y+x+\ell_k-1}(\bar{\eta}(z)-\bar{\eta}(z+1))(f(\eta)-f(\eta^{z,z+1}))\nu_\rho(d\eta).\label{eq:auxest33}
\end{align}
By Young's inequality we bound the first expression \eqref{eq:auxest3} above by
\begin{equation*}
\begin{split}
& \sum_{x\notin\Lambda_{y}^{\ell_k}}\sum_{z=y+x}^{y+x+\ell_k-1}\frac{v(x)A_x}{4\ell_k}\int (\tau_x\psi(\eta))^2(\bar{\eta}(z)-\bar{\eta}(z+1))^2\nu_\rho(d\eta)\\
+& \sum_{x\notin\Lambda_{y}^{\ell_k}}\sum_{z=y+x}^{y+x+\ell_k-1}\frac{v(x)}{4\ell_k A_x}\int(f(\eta)-f(\eta^{z,z+1}))^2\nu_\rho(d\eta).
\end{split}
\end{equation*}
By taking $4A_x=v(x)/n^2$ and doing similar estimates as above, we bound last expression by
\begin{equation}\label{error2.1}
\frac{C(\rho)}{n^2}\mathrm{Var}(\psi)\sum_{x\notin\Lambda_{y}^\ell}v^2(x)+\frac{n^2}{\ell}\sum_{x\notin\Lambda_{y}^\ell}\sum_{z=y+x}^{y+x+\ell-1}I_{z,z+1}(f).
\end{equation}
To bound the second term \eqref{eq:auxest33}, we use again Young's inequality and we bound it by
\begin{equation*}
\begin{split}
& \sum_{x\in\Lambda_{y}^{\ell_k}}\sum_{z=y+x}^{y+x+\ell_k-1}\frac{v(x)A_x}{4\ell_k\Xi^n_{z,z+1}}\int (\tau_x\psi(\eta))^2(\bar{\eta}(z)-\bar{\eta}(z+1))^2\nu_\rho(d\eta)\\
+&\sum_{x\in\Lambda_{y}^{\ell_k}}\sum_{z=y+x}^{y+x+\ell_k-1}\frac{v(x)\Xi^n_{z,z+1}}{4\ell_k A_x}\int(f(\eta)-f(\eta^{z,z+1})^2\nu_\rho(d\eta).
\end{split}
\end{equation*}
By taking $4A_x=v(x)/n^2$ and repeating the same arguments as in the previous lemma we bound last expression by
\begin{equation}\label{error2.2}
 C(\rho)\mathrm{Var}(\psi)\Big(\frac{n^\beta}{n^2\alpha \ell_k }+\frac{1}{n^2}\Big)\sum_{x\in\Lambda_{y}^{\ell_k}}v^2(x)+\frac{n^2}{\ell_k} \sum_{x\in\Lambda_{y}^{\ell_k}}\sum_{z=y+x}^{y+x+\ell_k-1}\Xi^n_{z,z+1}I_{z,z+1}(f).
\end{equation}
Putting together \eqref{error2.1} and \eqref{error2.2} we get the bound
\begin{equation*}
C(\rho)\mathrm{Var}(\psi)\bigg( \frac{1}{n}\|v\|_{2,n}^2+\frac{n^\beta}{n^2\alpha\ell_k}\sum_{x \in \Lambda_{y}^{\ell_k}}v^2(x)\bigg)
+\frac{n^2}{\ell_k} \sum_{x\in\mathbb{Z}}\sum_{z=y+x}^{y+x+\ell_k-1}\Xi^n_{z,z+1}I_{z,z+1}(f).
\end{equation*} 
Now, summing over  $y \in \{1,\ldots,\ell_k\},$ recalling \eqref{eq:init} and invoking Lemma \ref{lemma:dir1}, \eqref{eq:dirich3}, the proof ends. \end{proof}

Finally, last term \eqref{eq:dec3} is treated similarly as in Proposition \ref{prop:one-block-left}, as follows:

\begin{proposition} 
\label{prop:one-block3}
For any $\ell_0, L,M\in\bb {N}, $ such that $M\ge 1$, and any $t>0$:
\begin{multline}
\mathbb{E}_{\rho}\bigg[\Big(\int_{0}^t  \sum_{x\in\mathbb{Z}}v(x)\vec{\eta}^L_{sn^2}(x)\big(\vecleft{\eta}^{2^{M-1}\ell_0}_{sn^2}(x)-\vecleft{\eta}_{sn^2}^{\ell_0}(x)\big)\Big)^2\;ds\bigg]\\ \leq
C(\rho)t\bigg( \frac{\ell_0^2}{Ln}\|v\|_{2,n}^2+\frac{n^\beta\ell_0}{n^2\alpha L}\sum_{x\neq -1}v^2(x)\bigg).\label{eq:bigerror2}
\end{multline}
\end{proposition}
\begin{proof}
We omit its proof since the argument is the same as for Proposition \ref{prop:one-block-left}. Let us notice that the support of $\psi(\eta):=\vec\eta^L(0)$ does not intersect $\{-2^{M-1}\ell_0,\ldots,-1\}$, which is enough to make the proof work.
\end{proof}

Putting Proposition \ref{doub box} and \ref{prop:one-block3} together,  we now can reach the box of size $L \geq \ell_0$.

\begin{proposition}
\label{prop:final-size}
For any $\ell_0\leq L\in\bb {N}$ and  $t>0$:
\begin{multline}
\mathbb{E}_{\rho}\bigg[\Big(\int_{0}^t  \sum_{x\in\mathbb{Z}}v(x)\vecleft{\eta}^{\ell_0}_{sn^2}(x)\big(\vec{\eta}_{sn^2}^{\ell_0}(x)-\vec{\eta}_{sn^2}^{L}(x)\big)\;ds\Big)^2\bigg]\\
\leq C(\rho)t\bigg( \Big\{\frac{\ell_0^2}{Ln}+\frac{L}{n}   \Big\}\|v\|_{2,n}^2+\frac{n^\beta}{n^2\; \alpha}\Big\{\frac{\ell_0}{L}+(\log_2(L))^2\Big\}\sum_{x\neq -1}v^2(x) \bigg). \label{eq:bigerror3}
\end{multline}
\end{proposition}

\begin{remark}
Since $\ell_0$ is supposed to be independent of $n$, \eqref{eq:bigerror3} is also bounded by
\[
Ct \bigg(\frac{L}{n}\|v\|^2_{2,n} +\frac{n^\beta\;(\log_2(L))^2}{\alpha n^2}\sum_{x \neq -1} v^2(x)  \bigg).
\]
\end{remark}

\begin{proof}
We start by showing the result in the case where $L=\ell_M=2^M\ell_0$ for $M$ a positive integer.
We use the decomposition \eqref{eq:dec1}+\eqref{eq:dec2}+\eqref{eq:dec3}.

By the convexity inequality $(a+b+c)^2 \leq 3(a^2+b^2+c^2)$ and using the Minkowski's inequality twice, the expectation in the statement of the proposition is bounded from above by
\begin{align*}
&3\;\bigg\{\sum_{k=0}^{M-1}\bigg(\mathbb{E}_{\rho}\bigg[\Big(\int_{0}^t  \sum_{x\in\mathbb{Z}}v(x)\vecleft{\eta}^{\ell_k}_{sn^2}(x)\Big\{\vec{\eta}_{sn^2}^{\ell_k}(x)-\vec{\eta}_{sn^2}^{\ell_{k+1}}(x)\Big\}\; ds\Big)^2\bigg]\bigg)^{1/2}\bigg\}^2\\
& + 3\;\bigg\{\sum_{k=0}^{M-2}\bigg(\mathbb{E}_{\rho}\bigg[\Big(\int_{0}^t  \sum_{x\in\mathbb{Z}}v(x)\vec{\eta}^{\ell_{k+1}}_{sn^2}(x)\Big\{\vecleft{\eta}_{sn^2}^{\ell_k}(x)-\vecleft{\eta}_{sn^2}^{\ell_{k+1}}(x)\Big\}\; ds\Big)^2\bigg]\bigg)^{1/2}\bigg\}^2 \\
& +3\;\mathbb{E}_{\rho}\bigg[\Big(\int_{0}^t ds \sum_{x\in\mathbb{Z}}v(x)\vec{\eta}^L_{sn^2}(x)\big(\vecleft{\eta}^{\ell_{M-1}}_{sn^2}(x)-\vecleft{\eta}_{sn^2}^{\ell_0}(x)\big)\; ds\Big)^2\bigg].
\end{align*}
The last term in the previous expression can be bounded by Proposition \ref{prop:one-block3}. 
By Propositions \ref{doub box} and \ref{prop:one-block3}, assuming \[\tau_x\psi(\eta)=\vecleft\eta^{\ell_k}(x), \qquad \tau_x\widetilde\psi(\eta)=\vec\eta^{\ell_{k+1}}(x),   \] which have both a variance of order $C/\ell_k$, wa can deduce that the first two terms in the expression above are  bounded from above by
\begin{align*}
 C(\rho)t&\bigg(2\sum_{k=0}^{M-1}\Big(\frac{\ell_k}{n}\|v\|^2_{2,n}+\frac{n^\beta}{\alpha n^2}\sum_{x \neq -1} v^2(x)\Big)^{1/2}\bigg)^2\\
 &\leq  C(\rho)t\bigg(\sum_{k=0}^{M-1}\Big(\frac{\ell_k}{n}\|v\|^2_{2,n}\Big)^{1/2}+\Big(\frac{n^\beta}{\alpha n^2}\sum_{x \neq -1} v^2(x)\Big)^{1/2}\bigg)^2\\
& \leq C(\rho)t\frac{2}{n}\bigg(\sum_{k=0}^{M-1}2^{k/2}\ell_0^{1/2}\bigg)^2\|v\|^2_{2,n}+2M^2\frac{ n^\beta}{\alpha n^2}\sum_{x \neq -1} v^2(x)\\
& \leq C(\rho)t\bigg(\frac{L}{n} \|v\|^2_{2,n} +\frac{M^2n^\beta}{\alpha n^2}\sum_{x \neq -1} v^2(x)\bigg).
\end{align*}
Putting together the two previous bounds we obtain the result.
 In the other cases we choose $M$ sufficiently big such that $2^M\ell_0\leq L\leq 2^{M+1}\ell_0$ and a similar computation to the one above proves the claim.
\end{proof}

\subsection{Estimate of \eqref{eq:term5}}
\label{ssec:decomp}

\begin{proposition}\label{prop:decomp}
For any $ L\in\bb {N}$ and  $t>0$:
\begin{multline*}
\mathbb{E}_{\rho}\bigg[\Big(\int_{0}^t  \sum_{x\in\mathbb{Z}}v(x)\Big\{\bar{\eta}_{sn^2}(x)\vec{\eta}_{sn^2}^{L}(x)-(\vec{\eta}_{sn^2}^{L}(x))^2+\frac{1}{2L}
\big(\bar{\eta}_{sn^2}(x)-\bar{\eta}_{sn^2}(x+1)\big)\Big\}\; ds\Big)^2\bigg]\\ \leq
   C(\rho)t\Big(\frac{L}{n}+\frac{n^\beta}{n\alpha}\Big)\|v\|_{2,n}^2.
\end{multline*}
\end{proposition}
\begin{proof}
By \cite[Lemma 2.4]{KLO} and following the same arguments as in Proposition \ref{prop:one-block} we have to compute the $H_{-1}$ norm of the function in the statement of the proposition. 
For that purpose notice that
\begin{align}
&2\int \sum_{x\in\mathbb{Z}}v(x)\vec{\eta}^L(x)\Big\{\bar{\eta}(x)-\vec{\eta}^L(x)\Big\}f(\eta)\nu_\rho(d\eta)\label{eq:firstpart}\\
=\; &2\int \sum_{x\in\mathbb{Z}}v(x)\vec{\eta}^L(x)\Big\{\bar{\eta}(x)-\bar{\eta}(x+1)+
\frac{L-1}{L}\big(\bar{\eta}(x+1)-\bar{\eta}(x+2)\big)\notag\\
 &\hspace{3cm} +\cdots+\frac{1}{L}\big(\bar{\eta}(x+L-1)-\bar{\eta}(x+L)\big)\Big\}f(\eta)\nu_\rho(d\eta).\notag
\end{align}
Now we write the previous expression as
\begin{equation*}
\begin{split}
&2\int \sum_{x\in\mathbb{Z}}v(x)\vec{\eta}^L(x)\Big\{\bar{\eta}(x)-\bar{\eta}(x+1)\Big\}f(\eta)\nu_\rho(d\eta)\\
+\;&2\int \sum_{x\in\mathbb{Z}}v(x)\vec{\eta}^L(x)\frac{L-1}{L}\Big\{\bar{\eta}(x+1)-\bar{\eta}(x+2)\Big\}f(\eta)\nu_\rho(d\eta)\\
+\;&\cdots+ 2\int \sum_{x\in\mathbb{Z}}v(x)\vec{\eta}^L(x)\frac{1}{L}\Big\{\bar{\eta}(x+L-1)-\bar{\eta}(x+L)\Big\}f(\eta)\nu_\rho(d\eta).
\end{split}
\end{equation*}
In each one of the terms above, we write it as twice its half, and in one of the integrals we make the change $\eta$ to $\eta^{z,z+1}$ (for some suitable $z$), for which the measure $\nu_\rho$ is invariant. Thus, the last expression equals
\begin{align}
&\int \sum_{x\in\mathbb{Z}}v(x)\vec{\eta}^L(x)\Big\{\bar{\eta}(x)-\bar{\eta}(x+1)\Big\}\Big(f(\eta)-f(\eta^{x,x+1})\Big)\nu_\rho(d\eta) \label{eq:lem6101}\\
+&\int \sum_{x\in\mathbb{Z}}v(x)\vec{\eta}^L(x)\frac{L-1}{L}\Big\{\bar{\eta}(x+1)-\bar{\eta}(x+2)\Big\}\Big(f(\eta)-f(\eta^{x+1,x+2})\Big)\nu_\rho(d\eta) \notag\\
+&\cdots \\&+\int \sum_{x\in\mathbb{Z}}v(x)\vec{\eta}^L(x)\frac{1}{L}\Big\{\bar{\eta}(x+L-1)-\bar{\eta}(x+L)\Big\}\Big(f(\eta)-f(\eta^{x+L,x+L+1})\Big)\nu_\rho(d\eta)\notag\\
+&\int\sum_{x\in\mathbb{Z}}v(x)\frac{\bar{\eta}(x+1)-\bar{\eta}(x)}{L}\Big\{\bar{\eta}(x)-\bar{\eta}(x+1)\Big\}f(\eta)\nu_\rho(d\eta). \label{eq:lem6102}
\end{align}
Notice that the last term  \eqref{eq:lem6102} comes from the change of variables $\eta$ to $\eta^{x,x+1}$ in the first term \eqref{eq:lem6101} above.
The whole sum can be rewritten as
\begin{align}
&\int \sum_{x\in\mathbb{Z}}v(x)\vec{\eta}^L(x)\frac{1}{L}\sum_{y=x+1}^{x+L}\sum_{z=x}^{y-1}\Big\{\bar{\eta}(z)-\bar{\eta}(z+1)\Big\}\Big(f(\eta)-f(\eta^{z,z+1})\Big)\nu_\rho(d\eta)\label{eq:boundlem}\\
-&\int\sum_{x\in\mathbb{Z}}v(x)\frac{1}{L}\big(\bar{\eta}(x)-\bar{\eta}(x+1)\big)^2f(\eta)\nu_\rho(d\eta).\label{eq:secondpart}
\end{align}
The integral in the statement of the proposition  is exactly equal to the sum of \eqref{eq:firstpart} and \eqref{eq:secondpart}, therefore it is bounded by the first term in the previous expression, namely \eqref{eq:boundlem}.
As before, at this point we have to be careful and split the sum in $x$ according to whether the slow bond intersects the set of bonds
\[\widehat{\mathcal{E}}^L_x=\{(z,z+1): x\leq z\leq y-1\,\text{ and } \,x+1\leq y\leq x+L\}.\]
A simple computation shows that for $x\in\Lambda_{0}^L=\{-L,\ldots,-1\}$ the slow bond belongs to the set $\widehat{\mathcal{E}}_{x}^L$, otherwise, it does not. Then, the first term in last expression is equal to
\begin{equation*}
\begin{split}
&\int \sum_{x\notin\Lambda_{0}^L}v(x)\vec{\eta}^L(x)\frac{1}{L}\sum_{y=x+1}^{x+L}\sum_{z=x}^{y-1}\Big\{\bar{\eta}(z)-\bar{\eta}(z+1)\Big\}\Big(f(\eta)-f(\eta^{z,z+1})\Big)\nu_\rho(d\eta)\\
+&\int \sum_{x\in\Lambda_{0}^L}v(x)\vec{\eta}^L(x)\frac{1}{L}\sum_{y=x+1}^{x+L}\sum_{z=x}^{y-1}\Big\{\bar{\eta}(z)-\bar{\eta}(z+1)\Big\}\Big(f(\eta)-f(\eta^{z,z+1})\Big)\nu_\rho(d\eta).
\end{split}
\end{equation*}
Now, we use the same arguments as above. {In each of the terms above we use Young's inequality with $2A_x=L v(x)/n^2 $} and we bound the first term by
\begin{equation*}
C(\rho)\frac{L}{ n^2} \sum_{x\notin\Lambda_{0}^L}v^2(x)+\frac{n^2}{{L^2}} \sum_{x\notin\Lambda_{0}^L}\sum_{y=x+1}^{x+L}\sum_{z=x}^{y-1}I_{z,z+1}(f).
\end{equation*}
The second term is bounded by
\begin{equation*}
C(\rho)\bigg(\frac{n^\beta}{n^2\alpha} \sum_{x\in\Lambda_{0}^L}v^2(x)+\frac{L}{n^2}\sum_{x\in\Lambda_{0}^L}v^2(x)\bigg)+\frac{n^2}{{L^2}} \sum_{x\in\Lambda_{0}^L}\sum_{y=x+1}^{x+L}\sum_{z=x}^{y-1}\Xi^n_{z,z+1}I_{z,z+1}(f).
\end{equation*}
Putting together the two previous estimates plus Lemma \ref{lemma:dir1}, precisely \eqref{eq:dirich1} with $\ell=L$, the proof ends.
\end{proof}

\subsection{Estimate of \eqref{eq:term6}: Cauchy-Schwarz inequality}
\label{ssec:cs}
%

\begin{proposition}\label{prop:cs2}
For any $L\in\bb {N}$ and  $t>0$:
\begin{equation*}
\mathbb{E}^n_{\rho}\bigg[\Big(\int_{0}^t  \sum_{x\in\mathbb{Z}}v(x)
\Big\{\frac{(\bar{\eta}_{sn^2}(x)-\bar{\eta}_{sn^2}(x+1))^2}{2L}-\frac{\chi(\rho)}{L}\Big\}\;ds\Big)^2\bigg]\leq
 C(\rho)\frac{ t^2n}{L^2}\|v\|_{2,n}^2.
\end{equation*}
\end{proposition}
\begin{proof}
The proof is straightforward using the Cauchy-Schwarz inequality.
\end{proof}

\subsection{Technical lemma: estimates in the Dirichlet form}
\label{ssec:dirichlet}
\begin{lemma}\label{lemma:dir1}
Recall the definition \eqref{eq:def_i}. For any $\ell\in\mathbb{N}$ it holds that
\begin{align}
\frac{n^2}{\ell^2}\sum_{x\in\mathbb{Z}}\sum_{y=x+1}^{x+\ell}\sum_{z=x}^{y-1}\Xi^n_{z,z+1}I_{z,z+1}(f)& \leq n^2D_n(f), \label{eq:dirich1}\\
\frac{n^2}{\ell^2}\sum_{x\in\mathbb{Z}}\sum_{y=x-\ell}^{x-1}\sum_{z=y}^{x-1}\Xi^n_{z,z+1}I_{z,z+1}(f)& \leq n^2D_n(f),\label{eq:dirich2}\\
\frac{n^2}{\ell^2}\sum_{x\in\mathbb{Z}}\sum_{y=1}^\ell\sum_{z=y+x}^{y+x+\ell-1}\Xi^n_{z,z+1}I_{z,z+1}(f) &\leq n^2D_n(f).\label{eq:dirich3}
\end{align}
\end{lemma}
\begin{proof}
We present the proof for the first estimate \eqref{eq:dirich1}, but  one can do exactly the same argument for the other two ones.
By definition, we have
\begin{equation} D_n(f)=\sum_{x \in \bb Z} \Xi_{x,x+1}^n I_{x,x+1}(f),\label{eq:dirichlet}\end{equation}
where
\begin{equation}\label{eq:zeta}\Xi_{x,x+1}^n(\eta)=\begin{cases}  \frac{1}{2}+\frac{a}{2n^\gamma}; &  x \neq -1\\
\frac{\alpha}{2n^\beta}+\frac{a}{2n^\gamma};&  x =-1. \end{cases}
\end{equation}
Moreover, if $x \notin \Lambda_0^\ell$, then, by translation invariance of the measure $\nu_\rho$, for all $z$ in $\mathcal{E}_x^\ell$ (the set of bonds defined in \eqref{eq:set}) we have \[  \Xi_{z,z+1}^n I_{z,z+1}(f)= \Xi_{x+1,x+2}^n I_{x+1,x+2}(f). \]
From that observation, the quantity
\[\frac{n^2}{\ell^2}\sum_{x\notin\Lambda_0^\ell}\sum_{y=x+1}^{x+\ell}\sum_{z=x}^{y-1}\Xi^n_{z,z+1}I_{z,z+1}(f)\]
is bounded from above by
\begin{equation}\label{eq:dir1}
n^2\sum_{x\notin\Lambda_0^\ell}\Xi^n_{x+1,x+2}I_{x+1,x+2}(f).
\end{equation}
Now, if $x \in \Lambda_0^\ell$, we have to isolate the bond $\{-1,0\}$, that appears exactly $\ell^2$ times. For the other bonds, the same translation invariance argument holds. More precisely,
\begin{equation*}
\sum_{x\in\Lambda_0^\ell}\sum_{y=x+1}^{x+\ell}\sum_{z=x}^{y-1}\Xi^n_{z,z+1}I_{z,z+1}(f)= \ell^2 \Xi_{-1,0}^n I_{-1,0}(f) + \ell^2 \sum_{x \in \Lambda_0^\ell} \Xi^n_{x+1,x+2}I_{x+1,x+2}(f).\end{equation*}
Multiplying by $n^2/\ell^2$, and putting the previous expression together with \eqref{eq:dir1}, the proof ends.
\end{proof}

\section{Auxiliary results}
\label{sec:auxi}

We give here two auxiliary results which can be obtained very similarly, following the proof of Proposition \ref{prop:one-block}. We need the Corollary \ref{cor:estimate} (stated below) for the case $\beta<1$ in Subsection \ref{ssec:limitpoints} when we prove the energy estimate, namely \eqref{eq:energyestimate}. We need Proposition  \ref{prop:estim-1} below in Subsection \ref{vanishB}.

\begin{corollary}\label{cor:estimate}
Let  $v:\mathbb{Z}\to{\mathbb{R}}$ be a measurable function that satisfies \eqref{vsummable}. Then, there exists $C(\rho)>0$ such that for any $t>0$ and any $\ell,n \in \bb N$:
\begin{multline*}
\mathbb{E}_{\rho}\bigg[\Big(\sqrt n\int_{0}^t \sum_{x\in\ell \mathbb{Z}}v(x)\Big\{\vec{\eta}_{sn^2}^{\ell}(x)-\vec{\eta}_{sn^2}^{\ell}(x+\ell)\Big\}\; ds\Big)^2\bigg]\\
\leq
 C(\rho)t\;\bigg\{\frac{\ell}{n} \sum_{x \in \ell \bb Z} v^2(x)+\frac{ n^\beta}{\alpha n}\Big( v^2(-\ell) + v^2(-2\ell)\Big)\bigg\}.
\end{multline*}
In particular, assume $\beta < 1$ and let $\vphi\in\mc S(\bb R)$. If $\ell=\varepsilon n$ ($\varepsilon$ fixed), and $v(x)=\nabla \vphi(x/n)$, then
\begin{multline}
\limsup_{n\to\infty}\mathbb{E}_{\rho}\bigg[\Big(\sqrt n\int_{0}^t  \sum_{x\in\varepsilon n \mathbb{Z}}\nabla\vphi\Big(\frac{x}{n}\Big)\Big\{\vec{\eta}_{sn^2}^{\varepsilon n}(x)-\vec{\eta}_{sn^2}^{\varepsilon n}(x+\varepsilon n)\Big\}\; ds\Big)^2\bigg] \\ \leq
 C(\rho)t\varepsilon  \sum_{x \in \bb Z} \big(\nabla\vphi(\varepsilon x)\big)^2. \label{eq:bound_charact}
\end{multline}
\end{corollary}
\begin{proof}
As in the one-block estimate, namely, Proposition \ref{prop:one-block}, we need to bound the following norm
\begin{equation*}
n\sum_{y=1}^{\ell}\bigg\|\sum_{x\in\ell\mathbb{Z}}v(x)\frac{1}{\ell}(\bar{\eta}(y+x)-\bar{\eta}(y+x+\ell))\bigg\|_{-1}^2.\end{equation*}
Repeating the proof of Proposition \ref{doub box} (more precisely, one needs to bound almost the same quantity as  \eqref{eq:init}, except that the sum runs out of $\ell \bb Z$ and we cancel out the term $\tau_x\psi$), one can easily show that this quantity is bounded by
\[
C(\rho)n\sum_{y=1}^{\ell}\bigg(\frac{1}{n^2}  \sum_{x \in \ell \bb Z} v^2(x) +\frac{n^\beta}{\alpha n^2\ell}\sum_{x \in \Lambda_{y}^\ell}v^2(x)\bigg).\]
Notice that, for any $y \in \bb Z$, the set $\Lambda_{y}^\ell \cap \ell \bb Z$ only contains one element, which can be either $-\ell$ or $-2\ell$. Since $y$ runs over $\ell$ elements, the last quantity is bounded by
\[C(\rho)\bigg(\frac{\ell}{n} \sum_{x \in \ell \bb Z} v^2(x)+\frac{ n^\beta}{\alpha n}\Big\{ v^2(-\ell) + v^2(-2\ell)\Big\}\bigg).
\]
This ends the first part of the proof. Now, take $\ell=\varepsilon n$ ($\varepsilon$ fixed), and $v(x)=\nabla \vphi(x/n)$ with $\vphi \in \mc S(\bb R)$. The bound becomes
\[C(\rho)\bigg(\varepsilon \sum_{x \in \varepsilon n \bb Z} \Big(\nabla\vphi\Big(\frac{x}{n}\Big)\Big)^2+\frac{ n^\beta}{\alpha n}\Big\{ \big(\nabla\vphi(-\varepsilon)\big)^2 + \big(\nabla\vphi(-2\varepsilon)\big)^2\Big\}\bigg).
\]
Since, in the case $\beta <1$,
\[\frac{ n^\beta}{\alpha n}\big(\nabla\vphi(-\varepsilon)\big)^2 \xrightarrow[n \to \infty]{} 0,\]
 we have proved \eqref{eq:bound_charact}.
\end{proof}

\begin{proposition}\label{prop:estim-1}
There exists a constant $C>0$, such that for any $\epsilon\in[0,\frac{1}{4}]$, any $t>0$ and any $n\in\bb N$:
\begin{equation}\label{eq:-10}
\mathbb{E}^n_{{\rho}}\bigg[\Big(\int_{0}^t \bar\eta_{sn^2}(0)\bar\eta_{sn^2}(-1) \; ds\Big)^2\bigg] \le \frac{Ct}{n^{1+\epsilon}}.
\end{equation}
\end{proposition}

\begin{proof}
In order to prove the proposition we use  the following decomposition:
\begin{align}
\bar\eta(-1) \bar\eta(0)&= \bar\eta(0)\big(\bar\eta(-1)-\vecleft\eta^{\ell}(-1)\big)\label{eq:firstterm}\\
&\quad +\vecleft\eta^{\ell}(-1)\big(\eta(0)-\vec\eta^L(0)\big) \label{eq:secondterm}\\
&\quad +\vecleft\eta^{\ell}(-1)\vec\eta^L(0),\label{eq:thirdterm}
\end{align}
with $\ell,L\in \bb N$. Now we notice that by  \cite[Lemma 7.1]{fgn3} we have 
\begin{equation}\label{eq:easy1}
\mathbb{E}_{\rho} \bigg[ \Big( \int_0^t \bar\eta_{sn^2}(0)\big(\bar\eta_{sn^2}(-1)-\vecleft\eta_{sn^2}^{\ell}(-1)\big) \,ds \Big)^2\bigg]\leq Ct \; \frac{\ell}{n^{2}},
\end{equation}
and by a similar argument as in the proof of \cite[Lemma 7.1]{fgn3} one can easily obtain that
\begin{equation}\label{eq:easy2}
\mathbb{E}_{\rho} \bigg[ \Big( \int_0^t \vecleft\eta_{sn^2}^{\ell}(-1)\big(\eta_{sn^2}(0)-\vec\eta_{sn^2}^L(0)\big) \,ds \Big)^2\bigg]\leq Ct\frac{L}{n\ell}.
\end{equation}
Finally, by the Cauchy-Schwarz inequality, and the fact that the two empirical averages below do not intersect, we get
\[
\mathbb{E}_{\rho} \bigg[ \Big( \int_0^t \vecleft\eta_{sn^2}^{\ell}(-1)\vec\eta_{sn^2}^L(0)  \,ds \Big)^2\bigg]\leq C\frac{t^2}{\ell L}.
\]
Now we make the choice $\ell=Ln^\epsilon=n^{1-\epsilon}$ from which the result follows. 

 Let us notice that the estimates given in \eqref{eq:easy1} and \eqref{eq:easy2} can also be recovered as particular cases of Propositions \ref{prop:one-block} and \ref{prop:one-block-left}. 
\end{proof}

\appendix \section{Semi-group tools} \label{app}
In this appendix, we present an useful result on the semi-groups associated to the operators $\Delta_\beta$ defined on $\mc S_\beta(\bb R)$, namely, the condition \eqref{contingency}. This property was already needed in \cite{fgn3}, but not proved there.
 We start by recalling the three PDE's associated to the different regimes of $\beta$.  We also remark that the operator $\Delta_\beta$ is essentially the Laplacian operator in a specific domain.
 
\subsection{Regime $\beta\in[0,1)$}
The PDE associated to this regime is the heat equation on the line, or else,
\begin{equation}\label{pde1}
\left\{
\begin{array}{ll}
\partial_t u(t,x) = \; \frac{1}{2}\partial_{xx}^2 u(t,x), &t \geq 0,\, x \in \mathbb R,\\
u(0,x) = \; g(x), &x \in \mathbb R.
\end{array}
\right.
\end{equation}
It is a classical fact that the semi-group related to \eqref{pde1} is given by
\begin{equation}\label{sem heat eq}
  T_t g(x):= \frac{1}{\sqrt{2\pi t}}\int_{\bb R} e^{-\frac{(x-y)^2}{2t}}g(y)\,dy\,,\quad \textrm{for }x\in\bb R\,.
 \end{equation}
  If $g\in\mc S(\bb R)$ then $T_t g\in \mc S(\bb R)$ and consequently $\Delta T_t g\in \mc S(\bb R)$, which proves  \eqref{contingency}.

\subsection{Regime $\beta\in(1,\infty]$}
Here, the associated PDE is the heat equation with a boundary condition of Neumann's type at $x=0$ given by
\begin{equation}\label{pde3}
\left\{
\begin{array}{ll}
\partial_t u(t,x) = \; \frac{1}{2}\partial^2_{xx} u(t,x), &t \geq 0,\, x \in \mathbb R\backslash\{0\},\\
\p_x u(t,0^+)=\p_x u(t,0^-)=0,  &t \geq 0,\\
u(0,x) = \; g(x), &x \in \mathbb R.
\end{array}
\right.
\end{equation}
Its semi-group reads as
  \begin{equation}\label{sem heat eq neu}
  T_t^\textrm{Neu} g(x):=
  \begin{cases}
\displaystyle   \frac{1}{\sqrt{2\pi t}}\int_{0}^{+\infty}\Big[
e^{-\frac{(x-y)^2}{2t}}+e^{-\frac{(x+y)^2}{2t}}\Big]g(y)\,dy\,,\quad &\textrm{for }x>0\,,\\
\displaystyle  \frac{1}{\sqrt{2\pi t}}\int_{0}^{+\infty}\Big[
e^{-\frac{(x-y)^2}{2t}}+e^{-\frac{(x+y)^2}{2t}}\Big]g(-y)\,dy\,,\quad &\textrm{for }x<0\,.\\
\end{cases}
 \end{equation}
We claim that $\p_{xx}^2 T_t g$ is again solution of \eqref{pde3}, but with initial condition $\p_{xx}^2 g$, which immediately leads to \eqref{contingency} in this case. One way to see this is to check it directly by differentiating twice the expression \eqref{sem heat eq neu}. Otherwise, one can recall how \eqref{sem heat eq neu} is usually deduced in the literature: in the positive half-line, one has to extend the initial profile $g$ to an even function in the whole line, then make this even function evolves according to \eqref{sem heat eq}, the semi-group of heat equation in $\bb R$. Since the semi-group \eqref{sem heat eq} preserves  even functions, and a smooth even function has zero derivative at zero, we conclude that \eqref{sem heat eq neu} is the solution of \eqref{pde3} in the positive half-line. The same argument applies to the negative half-line. Moreover, an even smooth function has all null derivatives of odd order at zero. This easily implies that $\p_{xx}^2 T_t^{\textrm{Neu}}g$ is a solution of \eqref{pde3} with initial condition $\p_{xx}^2 g$,  leading to \eqref{contingency}.

\subsection{Regime $\beta=1$}
The PDE associated to this regime is the heat equation with a boundary condition of Robin's type at $x=0$ given by
\begin{equation}\label{pde2}
\left\{
\begin{array}{ll}
 \partial_t u(t,x) = \; \frac{1}{2}\partial^2_{xx} u(t,x), &t \geq 0,\, x \in \mathbb R\backslash\{0\},\\
\p_x u(t,0^+)=\p_x u(t,0^-)=\alpha\{u(t,0^+)-u(t,0^-) \},  &t \geq 0,\\
 u(0,x) = \; g(x), &x \in \mathbb R.
\end{array}
\right.
\end{equation}
Denote by $g_{\textrm{even}}$ (resp. $g_{\textrm{odd}}$) the even (resp. odd) parts of a function $g:\bb R\to \bb R$: for $x\in{\mathbb{R}}$,
\begin{equation*}
 g_{\textrm{even}}(x)=\frac{g(x)+g(-x)}{2}\quad \textrm{and} \quad g_{\textrm{odd}}(x)=\frac{g(x)-g(-x)}{2}\,.
\end{equation*}
 The semi-group associated to \eqref{pde2} has been obtained in \cite{fgn3} by symmetry arguments. Its expression is
\begin{equation*}
  \begin{split}
  & T_t^\alpha g(x)= \frac{1}{\sqrt{2\pi t}}\Bigg\{\int_{\bb R}
e^{-\frac{(x-y)^2}{2t}} g_{\textrm{{\rm even}}}(y)\,dy \\
    & + \int_x^{+\infty} e^{-2\alpha (z-x)} \int_0^{+\infty}
\Big[\Big(\frac{z-y+2\alpha t}{t}\Big)e^{-\frac{(z-y)^2}{2t}}+\Big(\frac{z+y-2\alpha t}{t}\Big)e^{-\frac{(z+y)^2}{2t}}\Big]\,
g_{\textrm{{\rm odd}}}(y)\, dy\, dz\,\Bigg\}\,,\\
  \end{split}
  \end{equation*}
\noindent for $x>0$ and
  \begin{equation*}
  \begin{split}
 & T_t^{\alpha} g(x)= \frac{1}{\sqrt{2\pi t}}\Bigg\{\int_{\bb R}
e^{-\frac{(x-y)^2}{2t}} g_{\textrm{{\rm even}}}(y)\,dy \\
    & - \int_{-x}^{+\infty} e^{-2\alpha (x+z)} \int_0^{+\infty}
\Big[\Big(\frac{z-y+2\alpha t}{t}\Big)e^{-\frac{(z-y)^2}{2t}}+\Big(\frac{z+y-2\alpha t}{t}\Big)e^{-\frac{(z+y)^2}{2t}}\Big]\,
g_{\textrm{{\rm odd}}}(y)\, dy\, dz\,\Bigg\}\,,\\
  \end{split}
  \end{equation*}
\noindent for $x<0$.
Here, more important than the formula above is the symmetry that leads to its deduction. Decomposing the initial condition $g$ in its odd and even parts, and  using a similar symmetry argument, one can figure out that
\begin{equation}\label{eq34}
 T_t^\alpha g(x)\;=\;
 \begin{cases}
  T_tg_{\textrm{even}}(x)+\tilde{T}_t^\alpha g_{\textrm{odd}}(x)\,,&\textrm{for } x>0\,,\\
T_tg_{\textrm{even}}(x)-\tilde{T}_t^\alpha g_{\textrm{odd}}(-x)\,,&\textrm{for } x<0\,,\\
  \end{cases}
\end{equation}
where $\tilde{T}_t^\alpha$ is the semi-group of the following partial differential equation in the half-line:
\begin{equation}\label{pde4}
\left\{
\begin{array}{ll}
\partial_t u(t,x) = \; \frac{1}{2}\partial^2_{xx} u(t,x), &t \geq 0,\, x >0,\\
\p_x u(t,0^+)= 2\alpha u(t,0^+),  &t \geq 0,\\
u(0,x) = \; g(x), &x >0.
\end{array}
\right.
\end{equation}
We claim now that $\p_{xx}^2 T^\alpha_t g$ is again a solution of \eqref{pde2} with initial condition $\p_{xx}^2 g$. Provided by \eqref{eq34}, it is enough to show that $\p_{xx}^2 \tilde{T}^\alpha_t g_{\rm odd}$ is again a solution of \eqref{pde4} with initial condition $\p_{xx}^2 g_{\rm odd}$. In other words, we must assure that differentiating twice (in space) a solution of \eqref{pde4} yields again a solution of \eqref{pde4} with the same boundary condition (but  different initial condition).
Denote by $u$ the solution of \eqref{pde4} and consider
\begin{equation}\label{A8}
v=2\alpha u-\p_x u\,,
\end{equation}
which   is the solution of the following equation
\begin{equation}\label{eq Dir}
\left\{
\begin{array}{ll}
\partial_t v(t,x) = \; \frac{1}{2}\partial^2_{xx} v(t,x), &t \geq 0,\, x >0,\\
v(t,0^+)= 0,  &t \geq 0,\\
v(0,x) = \; v_0(x), &x >0.
\end{array}
\right.
\end{equation}
with $v_0=2\alpha g-\p_x g$. Last equation is the heat
equation with a boundary condition of Dirichlet's type.  The semigroup $T^{\textrm{Dir}}_t v_0(x)$ associated to last equation is
given by
\begin{equation}\label{sem dir}
 T^{\textrm{Dir}}_tv_0(x):=\frac{1}{\sqrt{2\pi t}}\int_0^{{+\infty}}\Big[e^{-\frac{(x-y)^2}{2t}}-e^{-\frac{(x+y)^2}{2t}}\Big]v_0(y)\,dy\,.
\end{equation}
If we show that $\p_{xx}^2 T^{\textrm{Dir}}_t v_0$ is again a solution of \eqref{eq Dir}, solving the ODE \eqref{A8} we will conclude that $\p_{xx}^2 \tilde{T}^\alpha_t g_{\rm odd}$ is again a solution of \eqref{pde2}.

Expression \eqref{sem dir} is obtained by a symmetry argument analogous to the previous one. More precisely, given an initial condition in the half-line, we extend it to an odd function in the entire real line and then make it evolve according to \eqref{sem heat eq}. Since \eqref{sem heat eq} preserves odd functions, and any odd smooth function
vanishes at the origin, we conclude that \eqref{sem dir} is the solution of \eqref{eq Dir}. We point out that the second derivative at zero of a smooth odd function vanishes. Hence  $\p_{xx}^2 T^{\textrm{Dir}}_t v_0$ is a solution of \eqref{eq Dir}, which leads to \eqref{contingency}.

\section*{Acknowledgements}

We are very grateful to the anonymous referees and also Oriane Blondel, who significantly
helped to improve the present version of this paper.

This work benefited from the support of the project EDNHS
ANR-14-CE25-0011 of the French National Research Agency (ANR). PG thanks CNPq (Brazil) for support through the research project
``Additive functionals of particle systems'', Universal
n. 480431/2013-2, also thanks FAPERJ ``Jovem Cientista do Nosso Estado''
for the grant E-25/203.407/2014 and the Research Centre of Mathematics
of the University of Minho, for the financial support provided by
``FEDER'' through the ``Programa Operacional Factores de Competitividade
COMPETE'' and by FCT through the research project
PEst-OE/MAT/UI0013/2014. MS thanks CAPES (Brazil) and IMPA (Instituto de Matematica Pura e Aplicada, Rio de Janeiro) for the fellowship ``Bolsa de Excel\^encia''.

\bibliography{bibliography}
\bibliographystyle{plain}

\end{document}